\documentclass[11pt, letterpaper]{amsart}
\usepackage{amsmath,amssymb,mathrsfs, xcolor}
\usepackage{hyperref}
\hypersetup{linkbordercolor = {white}, citebordercolor = {white},}
\usepackage[alphabetic]{amsrefs}

\usepackage[alphabetic]{amsrefs}

\numberwithin{equation}{section}

\theoremstyle{definition}
\newtheorem{thm}[equation]{Theorem} 
\theoremstyle{definition}
\newtheorem{definition}[equation]{Definition}
\newtheorem{prop}[equation]{Proposition} 
\newtheorem{cor}[equation]{Corollary}

\newtheorem{lemma}[equation]{Lemma}
\newtheorem{remark}[equation]{Remark}
\newtheorem{conj}[equation]{Conjecture}
\newtheorem{claim}[equation]{Claim}

\newcommand{\ovl}[1]{\overline{#1}}
\newcommand{\udl}[1]{\underline{#1}}
\newcommand{\pair}[1]{\langle #1\rangle}
\newcommand{\set}[1]{\left\{#1\right\}}

\newcommand{\pr}[1]{\left(#1\right)}
\newcommand{\abs}[1]{\left|#1\right|}
\newcommand{\bb}[1]{\mathbb{#1}} \newcommand{\td}[1]{\widetilde{#1}}

\newcommand{\n}{\nabla}

\newcommand{\sbst}{\subseteq}
\newcommand{\h}{\textbf H}

\newcommand{\bd}{\partial}

\newcommand{\N}{\mathbf{n}}
\newcommand{\C}{\mathcal C}

\title{Arnold--Thom conjecture for the arrival time of surfaces}

\author{Tang-Kai Lee}
\author{Jingze Zhu}

\address{MIT, Dept. of Math., 77 Massachusetts Avenue, Cambridge, MA 02139-4307}
\email{tangkai@mit.edu and zhujz@mit.edu}

\date{\today}

\begin{document}
\begin{abstract}
    Following \L ojasiewicz's uniqueness theorem and Thom's gradient conjecture, Arnold proposed a stronger version about the existence of limit tangents of gradient flow lines for analytic functions. 
    We prove \L ojasiewicz's theorem and Arnold's conjecture in the context of arrival time functions for mean curvature flows in $\bb R^{n+1}$ with neck or non-degenerate cylindrical singularities.
    In particular, we prove the conjectures for all mean convex mean curvature flows of surfaces, 
    including the cases when the arrival time functions are not $C^2.$ 
    The results also apply to mean curvature flows starting from two-spheres or generic closed surfaces.
\end{abstract}
\maketitle


\section{\bf Introduction}\label{sec:intro}
In \cite{T89}, Thom conjectured that if a gradient flow line for an analytic function has a limit point, then the secants of this gradient flow line converge.
Later, Arnold \cite{A04} proposed a stronger conjecture suggesting the convergence of its tangent directions.
This is now known as the Arnold--Thom gradient conjecture.

\begin{conj}
    [Arnold--Thom]\label{conj:AT}
    If a gradient flow line $\gamma(t)$ for an analytic function has a limit point, then the limit of $\frac{\gamma'(t)}{\abs{\gamma'(t)}}$ exists.
\end{conj}

Thom's conjecture is based on the \L ojasiewicz theorem about the zero-order asymptotic behavior of gradient flow lines.
In \cite{L84}, \L ojasiewicz proved that if a gradient flow line of an analytic function has a limit point, then the flow line has finite length, and hence has a unique limit point.
His proof used two important inequalities for analytic functions, now known as the \L ojasiewicz inequalities \cite{L}.
Based on the work of \L ojasiewicz, Kurdyka--Mostowski--Parusi\'nski proved Thom's conjecture in \cite{KMP} by showing that the curve of secants has finite length after being projected to the unit sphere.
However, Conjecture~\ref{conj:AT}, carrying finer first-order asymptotic information of gradient flow lines, is still open.

In this paper, we study the \L ojasiewicz theorem and the Arnold--Thom conjecture for solutions to the arrival time equation
\begin{align}\label{arv-eq}
	-1 = |\n u|\, {\rm div}\pr{\frac{\n u}{|\n u|}}
\end{align}
on mean convex domains in $\bb R^{n+1}.$
This is a degenerate elliptic partial differential equation, so its solutions are far from being analytic.
By the works of Evans--Spruck \cite{ES} (cf. \cite{CGG}) and Colding--Minicozzi \cite{CM16}, this equation admits a unique $C^{1,1}$ solution in any smooth compact mean convex domain given a Dirichlet boundary condition.
This equation is connected to mean curvature flows in the following geometric way.
If $u$ is an arrival time function in a mean convex domain $\Omega\sbst\bb R^{n+1},$ that is, $u$ solves the equation \eqref{arv-eq} in $\Omega,$ then its level set $M_t := \set{x\in\Omega: u(x)=t}$ evolves by the mean curvature flow.
In other words, the mean curvature vector $\h = -H\N$ of $M_t$ satisfies
$\bd_t x = \h.$
Here, a smooth domain is mean convex if its boundary has non-negative mean curvature with respect to its outer unit normal.

\subsection{Main results}
\label{sec:main-results}

Consider a gradient flow line for the arrival time function $u$, which is a solution to \eqref{arv-eq}. 
In this context, we prove the \L ojasiewicz theorem and solve the Arnold--Thom conjecture for arrival time functions in $\bb R^3$ (cf. \cite[Conjecture~1.4]{CM19}).

\begin{thm}
\label{thm:main}
    If $\gamma$ is a gradient flow line for the arrival time function of a bounded mean convex domain with a limit point,
    then the limit point is unique and the unit tangents $\frac{\gamma'(t)}{\abs{\gamma'(t)}}$ of $\gamma$ converge.
\end{thm}

This theorem is a consequence of a more general result that we show in $\bb R^{n+1}$ with $n\geq 2$.
In fact, we prove the \L ojasiewicz theorem and the Arnold--Thom conjecture when a limit point of a gradient flow is a neck singularity. 
The precise statement is recorded in Theorem~\ref{thm:AT-neck}.
The condition of Theorem~\ref{thm:AT-neck} is automatically true if the initial closed hypersurface $M\sbst\bb R^{n+1}$ satisfies one of the following conditions (cf. \cite{HS09,CM162,HK2,G24}).\\
(1) $M$ is two-convex.\\
(2) The entropy of $M$ satisfies $\lambda(M) \leq \lambda\pr{S^{n-2}\times\bb R^2}.$\\
Thus, we prove the \L ojasiewicz theorem and solve the Arnold--Thom conjecture in $\bb R^{n+1}$ in the two-convex case and the low-entropy case.

Another situation that we can work on in the higher dimensional case is the non-degenerate condition, which is also exploited in \cite{SS}.
A cylindrical singularity is called non-degenerate if, roughly speaking, the flow looks like a cylinder when zoomed in in any directions and scales.
That is, all limit flows\footnote{In the study of geometric flows, limiting models are described using the notion of limit flows and tangent flows.
See Definition~\ref{def:lim-flow}.}
are cylinders.
A priori, this assumption does not exclude the possibility that under different scales and directions, the flow looks like cylinders of different sizes or axes.
We prove that this cannot happen.
Using this result, we then prove the \L ojasiewicz theorem and the Arnold--Thom conjecture for gradient flow lines having a non-degenerate cylindrical singularity as a limit point.

\begin{thm}\label{thm:nondeg-cyl}
    Let $M_t$ be a closed mean convex mean curvature flow in $\bb R^{n+1}.$ 
    If $\gamma$ is a gradient flow line for the arrival time function of $M_t$ with a non-degenerate limit point,
    then the limit flow is unique, the limit point is unique, and the unit tangents of $\gamma$ converge.
\end{thm}

Examples of non-degenerate cylindrical singularities were constructed \cite{GS09}.
The stability result of Schulze--Sesum \cite{SS} implies that there are more.
A related notion of non-degeneracy was discussed in \cite{SX4, SWX}.

Finally, we deal with general mean curvature flows with neck singularities when mean curvatures change signs or even when their level set flows fatten.
The key is the mean convex neighborhood theorem proven by Choi--Haslhofer--Hershkovits--White \cite{CHHW} (cf. \cite{CHH}), which implies that flows behave nicely near neck singularities.
This means that no matter how strangely a flow line goes, it will be constrained by the ordinary arrival time equation as in the mean convex case when the flow line approaches a neck singularity.
Therefore, we can apply the arguments in the mean convex case to get the convergence of tangents in this general setting.
In particular, with the recent breakthrough of Bamler--Kleiner \cite{BK23} about the multiplicity one conjecture, the \L ojasiewicz theorem and the Arnold--Thom conjecture are true for a mean curvature flow starting from a two-sphere (based on \cite{W95, Br16, HW, CHH}) or a generic surface in $\bb R^3$ (based on \cite{CCMS, CCS, SX1, SX2}).

\begin{cor}
    Let $M_t$ be a mean curvature flow starting from a two-sphere or a generic closed embedded surface in $\bb R^{3}.$
    If $\gamma$ is a flow line of $M_t$ with a limit point,
    then the limit point is unique and the unit tangents of $\gamma$ converge.
\end{cor}

As we will see later, solutions to the arrival time equation \eqref{arv-eq} may not be $C^2,$ let alone analytic.
Without mean convexity, solutions to the level set equations are even known to be only Lipschitz continuous.
Thus, the results above support the principle ``that solutions of degenerate equations behave as though they are analytic even when they are not'' proposed by Colding--Minicozzi \cite{CM19}.

\subsection{Previous results and proof ideas}
We mention some earlier works.
The method of using level sets to study geometric flows started from Osher--Sethian \cite{OS88}, Evans--Spruck \cite{ES}, Chen--Giga--Goto \cite{CGG}, and Ilmanen \cite{I94}.
In the case of mean convex mean curvature flows, Evans--Spruck further showed that level set flows can be given by functions with a separation of variables, leading to the arrival time equation \eqref{arv-eq}.
They proved that there is always a Lipschitz viscosity weak solution to this degenerate elliptic partial differential equation by approximating it using solutions to perturbed elliptic equations.

The optimal regularity of solutions to \eqref{arv-eq} has been a subtle issue for a long time.
When the singularity is spherical, based on the works of Huisken \cite{H93}, Sesum \cite{S08}, and Strehlke \cite{S20}, we know that the optimal regularity is $C^2$ when the dimension of flows is at least two.\footnote{For the curve case, the solutions to \eqref{arv-eq} are known to be at least $C^3$ based on the result of Kohn--Serfaty \cite{KS06}.
}
This turns out to be false when the singularity is not spherical.
In fact, Ilmanen \cite{I92} proved that even for a typical dumbbell example constructed by Grayson \cite{G89}, its arrival time function is not $C^2$ at the neck singularity.
In \cite{CM16, CM18}, Colding--Minicozzi improved the result and proved that solutions to \eqref{arv-eq} are always $C^{1,1}$ and twice differentiable.
That is, for an arrival time function $u,$ its Hessian exists everywhere, though it may not be continuous.
Moreover, at a cylindrical singularity $p$ modeled on $S^{n-k}\pr{\sqrt{2(n-k)}}\times \bb R^k$ of a mean convex flow $M_t,$ its arrival time function is $C^2$ at $p$ if and only if the singular set near the point is a $k$-dimensional $C^1$ submanifold (cf. Theorem~\ref{thm:C2} and \cite{G24}).
In both spherical and cylindrical cases, Sun--Xue \cite{SX4, SX3} showed that these low-regularity phenomena are generic. 
There are only few examples with $C^2$ regularity, for instance, the  $O(k+1)\times O(n-k+1)$ symmetric marriage ring $S^k\times S^{n-k}.$

For a $C^2$ arrival time function, in \cite{CM19}, Colding--Minicozzi proved that its gradient flow lines satisfy the Arnold--Thom conjecture.
In this delicate work, they first proved an optimal \L ojasiewicz inequality for a $C^2$ arrival time function near its singularity.
This is already a non-trivial result, implying the \L ojasiewicz theorem about the uniqueness of limits points for gradient flow lines.
Based on this, they reduced the Arnold--Thom conjecture to a sharp convergence rate for rescaled mean curvature flows converging to a cylinder.
Obtaining such a sharp convergence rate is technical, requiring a careful analysis of the kernel of the stability operator (cf. \cite{CM182}).

In general, solutions without $C^2$ regularity near cylindrical singularities are more prevalent. Therefore, we need to pursue new methods. 
We mention the proof ideas and how our results are related to previous results.
The first subtle point is that when the arrival time function $u$ is not $C^2,$ a priori, we don't even know the uniqueness of limit points, that is, the \L ojasiewicz theorem, as there is no \L ojasiewicz inequality for $u.$
In fact, if the uniqueness is not known, it does not make sense to formulate Thom's original conjecture.
We obtain this uniqueness result by applying the structure theorems for singularities established by Colding--Minicozzi \cite{CM162, CM18}.
This method uses special properties of one-dimensional curves and does not extend to singularities modeled on cylinders with more than one Euclidean factor.

Our theorems about the Arnold--Thom conjecture are based on two important ingredients. 
The first one is the regularity  of arrival time functions \cite{CM18, CM19}.
Given a mean convex flow $M_t$ having a neck singularity $p,$ to resolve the Arnold--Thom conjecture, we need to deal with the case when the arrival time function $u$ is not $C^2$ near $p.$
In this case, although there is no optimal regularity, using the structural result of singular sets \cite{CM162}, Colding--Minicozzi \cite{CM18} were able to show that the function $u$ still behaves nicely when approached away from the kernel of the Hessian of $u,$ that is, the axis of the tangent cylinder.
Thus, if a gradient flow line approaches a neck singularity transversely, one would expect the asymptotic behavior to be still nice.
This will be related to finer study near a neck singularity, which will be addressed next.

The second ingredient is the thorough understanding of the flow behavior near neck singularities.
By the works of Brendle--Choi \cite{BC19, BC21}, Angenent--Daskalopoulos--Sesum \cite{ADS20} and Choi--Haslhofer--Hershkovits--White \cite{CHHW}, the canonical neighborhoods near a neck singularity are fully classified:
along a gradient flow line, 
a limit flow can only be the shrinking cylinder or the translating bowl soliton constructed in \cite{AW} (cf. \cite{W03, W11, H15}) up to rotations and translations.
If limit flows are all shrinking cylinders along a gradient curve, then we can argue that the gradient flow line approaches the neck singularity transversely, so the situation is reduced to the work of Colding--Minicozzi.
Thus, the task becomes classifying possible limit flows along a fixed gradient flow line.
What we finally obtain is a uniqueness result of limit flows.
That is, along a fixed gradient flow line, if one limit flow is a shrinking cylinder, then all the others must be.
Once we have this uniqueness result, we can further prove the corresponding uniqueness when the limit flows are translating bowl solitons.
In this case, the flow approaches the singularity slowly and non-transversely, so we have to deal with it separately.

We remark that there are many examples of neck singularities at which their arrival time functions are not $C^2.$
As we mention, Sun--Xue \cite{SX4} proved that this is a generic phenomenon (cf.~\cite{SS}), and there are many explicit constructions \cite{AV, AAG}.
There are also works about infinite dimensional versions of Thom's conjecture, see \cite{CH23}.

We organize the paper as follows.
Section \ref{sec:prelim} provides a preliminary introduction to mean convex mean curvature flows.
In Section \ref{sec:AT-neck}, we first show the uniqueness of limit points for a gradient flow line with a neck limit point,
proving the \L ojasiewicz theorem.
We then prove the uniqueness of limit flows along such a gradient flow line and use it to prove the Arnold--Thom conjecture for neck singularities, which implies Theorem~\ref{thm:main}. 
In Section \ref{sec:gen-cyl}, we deal with general cylindrical singularities that are non-degenerate and prove Theorem~\ref{thm:nondeg-cyl}.
Section \ref{sec:general-flow} concludes with the results about general mean curvature flows with neck singularities, including flows starting from two-spheres and generic closed hypersurfaces.

\subsection*{\bf Acknowledgment}
The authors are grateful to Prof. Bill Minicozzi for several insightful comments.
They would like to thank Ao Sun and Zhihan Wang for helpful conversations and thank the anonymous referees for  constructive comments and suggestions, which have helped improve this paper.
During the project, Lee was partially supported by NSF Grant DMS 2005345.


\section{\bf Preliminary facts for mean curvature flows}
\label{sec:prelim}

In this section, we collect and prove some facts about mean convex mean curvature flows.

\subsection{Mean convex flows}
In the work of Chen--Giga--Goto \cite{CGG} and Evans--Spruck \cite{ES}, they showed the existence and uniqueness of viscosity solutions to the level set flow equation
\begin{align}\label{LSF-eq}
	\bd_t v = |\n v|\, {\rm div}\pr{\frac{\n v}{|\n v|}}
\end{align}
in a compact set in $\bb R^{n+1}.$
When a given hypersurface $M$ is mean convex, Evans--Spruck \cite{ES} further showed that the solution to \eqref{LSF-eq} can be given by a separation of variables, that is, of the form $v(x,t) = u(x) - t.$
Thus, the new function $u$ in this context satisfies the arrival time equation~\eqref{arv-eq}, and the mean convex level set flow $M_t$ starting from $M$ can be given by the level sets of $u,$ which is defined on the smooth closed mean convex domain $\Omega$ with $\bd\Omega = M.$
If we let $t=-T$ be the starting time, since $M_{-T}=M$ is closed and the flow is in a Euclidean space, all singularities occur in finite time.
Thus, the arrival time function $u$ constructed above has a finite supremum.
That is, there exists $T_{\rm large}>0$ such that $M_t=\emptyset$ for all $t\ge T_{\rm large}.$
Thus, for a closed mean convex level set flow, we will always assume the time interval to be $[-T,\infty).$

Closed mean convex level set flows have been studied extensively, cf. \cite{ES, HS99, W00, W03, W15, CM16, HK1}.
Following \cite{CM16}, we summarize some important properties of them in the following theorem.

\begin{thm}[\cite{ES, W00, W15, CM16}] \label{thm:mcMCF-u}
	Let $M_t$ ($t\in[-T,\infty)$) be a mean convex level set flow starting from a smooth closed domain $\Omega$ with $\bd\Omega=M_{-T}.$\\
	(1) There exists a unique $C^{1,1}$ viscosity solution $u\colon \Omega\to\bb R$ to the arrival time equation \eqref{arv-eq} on $\Omega$ with $u^{-1}(-T)=M_{-T}.$\\
	(2) The flow is non-fattening (i.e., the level sets have no interior), each level set is the boundary of a compact mean convex set, and $u$ is defined everywhere on the interior of $\Omega.$\\
	(3) Each tangent flow at a singularity is a multiplicity one generalized cylinder.\\
	(4) On the regular set for the flow, the function $u$ is smooth and $H_{M_t}= \frac 1{|\n u|}.$
\end{thm}

A (generalized) cylinder is one of the shrinkers $S^{n-k}\pr{\sqrt{2(n-k)}}\times \bb R^{k}$
for $k=0,\cdots,n.$
The definition of tangent flows is given in the following.
We will also use the notion of limit flows, which encode more information of a singularity.

\begin{definition}\label{def:lim-flow}
	Let $M_t\sbst\bb R^N$ be a mean curvature flow with a singularity at $(x,t)\in \bb R^N\times \bb R.$
	Given $\lambda>0$ and $(x_0, t_0)\in \bb R^N\times\bb R,$ we can translate and parabolically rescale the flow, producing a new mean curvature flow 
	\begin{align*}
	M^{(x_0, t_0,\lambda)}_t:=\lambda \pr{M_{\lambda^{-2}t + t_0} - x_0}.
	\end{align*} 
	(1) (Limit flow) If for some sequences $\lambda_i\to\infty $ and $(x_i, t_i)\to (x,t),$ the sequence of flows $M^{(x_i, t_i, \lambda_i)}_t$ (weakly) converges to a flow $M_t^\infty,$ then such a limit is called a limit flow at $(x,t).$\\
    (2) (Special limit flow) In the convergence above, if $t_i<t$ for all $i,$ then such a limit $M_t^\infty$ is called a special limit flow at $(x,t).$\\
	(3) (Tangent flow) In the convergence above, if $(x_i, t_i)=(x,t)$ for all $i,$ then such a limit $M_t^\infty$ is called a tangent flow at $(x,t).$
\end{definition}

We sometimes use the notation $\C_k$ to denote the space of all cylinders with a $k$-dimensional Euclidean factor.\footnote{Note that the same notations were used in \cite{CM15, CM162}, but in \cite{CM15}, $k$ means the spherical dimension.
Here, as in \cite{CM162}, Euclidean factors matter more when analyzing singularity structures.}
That is, we let
\begin{align*}
    \C_k:=\set{
    \text{all rotations and translations of }
    S^{n-k}\pr{\sqrt{2(n-k)}}\times \bb R^k
    \text{ in }\bb R^{n+1}
    },
\end{align*}
and $\C_k^0$ be the set of those cylinders in $\C_k$ centered at the origin (i.e., those whose axes pass through the origin).
A cylinder in $\C_k$ is sometimes called a $k$-cylinder.
These will be used in particular when we do not want to specify the direction of concerned cylinders.

Cylindrical singularities are the only possible singularities along a mean convex flow.
As mentioned in Theorem \ref{thm:mcMCF-u}, the arrival time function at a cylindrical singularity may not be $C^2.$
In~\cite{CM18}, this regularity property is nicely characterized by the local behavior of the singular set of the flow.
The localized version was addressed by Guo~\cite{G24}.
We record them in the following theorem.

\begin{thm}[\cite{CM18, G24}]
\label{thm:C2}
    Let $M_t$ be a mean convex level set flow starting from a smooth closed hypersurface.
    If $u$ is its arrival time function, $\mathcal S$ is its singular set, and $p$ is a singularity modeled on a $k$-cylinder, then $u$ is $C^2$ near $p$ if and only if there exists $\delta>0$ such that $\mathcal S\cap B_\delta(p)$ is a $k$-dimensional $C^1$ embedded submanifold.
\end{thm}

Next, we recall some definitions from \cite{CHH, CHHW}.
We use $B_r(x)$ to denote the open ball centered at $x\in\bb R^{n+1}$ with radius $r,$ and sometimes we write $B_r:=B_r(0)$ when the center is the origin.
We also write $H(x_0,t_0):=H_{M_{t_0}}(x_0)$ when $x_0\in M_{t_0}.$

\begin{definition}\label{def:sing}
    Let $M_t$ be a mean curvature flow, or more generally, a level set flow.
    \\
    (1) (Neck singularity) A singularity $(x_0, t_0)$ of the flow is called a cylindrical singularity if the tangent flow at $(x_0, t_0)$ is a multiplicity one $k$-cylinder.
    When $k=1,$ it is called a neck singularity.\\
    (2) (Neck cylindricality) For $\varepsilon>0,$ we say $(x_0, t_0)$ is $\varepsilon$-cylindrical if the rescaled flow $M^{\pr{x_0, t_0, H(x_0,t_0)}}_t$ is $\varepsilon$-close to a shrinking cylinder $S^{n-1}\pr{\sqrt{-2(n-1)t}}\times \bb R$ in $C^{\lceil 1/\varepsilon \rceil}\pr{B_{1/\varepsilon}\times [-\varepsilon^{-1},-\varepsilon]}$ after a parabolic rescaling, a translation in space and time, and a rotation.\footnote{This means that it can be written as the graph of a function over a cylinder in the region with $C^{\lceil 1/\varepsilon \rceil}$-norm at most $\varepsilon.$
    } 
\end{definition}

Cylindrical tangent flows are unique by \cite{CM15}.
That is, if a tangent flow at a singularity is given by one generalized cylinder, then any other tangent flow at the singularity is given by the same cylinder (with the same axis).
Thus, the definition above makes sense.

\subsection{Neck singularities}

Near a neck singularity, possible canonical neighborhoods are known \cite{BC19, ADS20, BC21, CHH, CHHW}.
They are shrinking spheres, shrinking cylinders, translating bowl solitons, or ancient ovals.
It turns out that any limit flow along a flow line will not be an ancient oval or a shrinking sphere if the flow line converges to a neck singularity.

\begin{lemma}\label{lem:no-bowl}
    Let $M_t$ be a mean convex mean curvature flow, $u$ be its arrival time function, and $\gamma(t)$ be a gradient flow line of $u$ converging to a singular point $p$ at time $0,$ that is, $u(p)=0.$
	Suppose $\gamma$ is parametrized by the time of the flow.
    If there are sequences $t_i\to 0$ and $\lambda_i\to\infty$ such that the sequence of rescaled flows 
	$M^{t_i}_t
    := M^{\pr{\gamma(t_i), t_i, \lambda_i}}_t
    $
	converges to an ancient oval or a shrinking sphere, then the limit point of the flow line is a spherical singularity.
\end{lemma}

The arguments to get Lemma~\ref{lem:no-bowl} are based on the definition of limit flows and Huisken's sphere theorem for convex hypersurfaces.
The proof ideas or similar arguments were sketched or used before (cf. \cite{BC19, BC21, CHH21}).

\begin{proof}
[Proof of Lemma~\ref{lem:no-bowl}]
	Suppose the sequence $M^{i}_t$ converges to an ancient oval that shrinks to a round point at $t=0.$
	We look at the time slice when $t=-1.$
	The convergence implies that as $i\to\infty,$ the sequence of slices $M^{t_i}_{-1} = \lambda_i \pr{M_{t_i - \lambda_i^{-2}} - \gamma(t_i)}$   
	converges to the time slice of the oval at $t=-1,$ which is a closed and strictly convex hypersurface.
	Hence, when $i$ is sufficiently large, the connected component of $M^{t_i}_{-1}$ containing $\gamma\pr{t_i - \lambda_i^{-2}}$ is also closed and strictly convex. 
	(Here, we may take larger $i$ such that $\gamma(t_i)$ and $\gamma\pr{t_i - \lambda_i^{-2}}$ are close.)
	Undoing the rescaling means that the connected component of $M_{t_i - \lambda_i^{-2}}$ containing $\gamma\pr{t_i - \lambda_i^{-2}}$ is also closed and strictly convex.
	This proves the lemma by Huisken's theorem \cite{H84}.
    The same arguments also work for the case of shrinking spheres.
\end{proof}

\begin{remark}\label{rmk:gamma-para}
    We may use different parametrizations for different uses in this paper.
    The ordinary gradient flow equation for $u$ is $\gamma'(t) = \n u \pr{\gamma(t)}.$
    When we say $\gamma$ is parametrized by the time of the flow, we use the reparametrization such that
    \begin{align*}
        \gamma'(t)
        = \frac{\n u}{|\n u|^2}\pr{\gamma(t)}
        = \h_{M_t} \pr{\gamma(t)}.
    \end{align*}
    When we say $\gamma$ is parametrized by its arc length, we use the reparametrization such that
    \begin{align*}
        \gamma'(t)
        = \frac{\n u}{|\n u|}\pr{\gamma(t)}.
    \end{align*}
    Throughout the paper, we usually parametrize the flow line \textit{by the time of the flow} unless otherwise specified.
\end{remark}

We can now state the following convergence result for a sequence of flows obtained by rescaling at a sequence of points along a gradient flow line of an arrival time function.

\begin{prop}
\label{prop:convergence}
	Let $M_t$ be a mean convex mean curvature flow, $u$ be its arrival time function, and $\gamma(t)$ be a gradient flow line of $u$ converging to a singular point $p$ with $u(p)=0.$
	If $p$ is a neck singularity, then any limit flow along $\gamma$ is a shrinking cylinder or a translating bowl with its axis \textit{parallel} to that of the tangent flow.
	To be more specific, if $\gamma$ is parametrized by time, then given any $t_i\to 0,$ after passing to a subsequence, the sequence of flows
	$M^{t_i}_t
    := M^{\pr{\gamma(t_i), t_i, \lambda_i}}_t,
    $
	where $\lambda_i = H\pr{\gamma(t_i), t_i},$ converges to a non-collapsed ancient solution, which is either a shrinking cylinder or a translating bowl soliton with its axis parallel to that of the cylindrical tangent flow.
\end{prop}

Since this result is essential in the rest of the paper, we briefly outline how it follows from the previous results mentioned above.
One can also see Appendix~\ref{sec:CHHW+SS} for a complete proof.

\begin{proof}
[Proof outline of Proposition~\ref{prop:convergence}]

    The proposition follows directly from 
    Lemma~\ref{lem:no-bowl} and \cite[Proposition 9.1]{CHHW} (cf. \cite[Proposition~2.3]{SS}).
    Here, we use $H(x,t)$ to rescale the flow and get a blowup limit instead of using $1/R(x,t)$ in \cite{CHHW} since they are equivalent in the setting of mean convex flows based on the curvature estimates \cite[Theorem~1.8]{HK1}.

    For the fact about parallel axes, it was first mentioned at the end of the proof outline \cite[Section~1.7]{CHHW}.
    A proof can be found in the proof of \cite[Proposition~2.3(iv)]{SS}. 
    This improvement is based on the fact that in the definition of the cylindrical scale in \cite{SS}, 
    the information of axes is kept.
    
    The proof in \cite{SS} essentially provides what we outline above.
    To be self-contained, a complete proof is given in Appendix~\ref{sec:CHHW+SS} for interested readers.
\end{proof}

A subtle point in Proposition~\ref{prop:convergence} is that we assume the convergence of $\gamma,$ which is stronger than the assumption of just having a limit point.
This will not be an issue since later we will first prove the \L ojasiewicz theorem for such a flow line (Theorem~\ref{thm-L}), and hence having a limit point implies that the flow line converges.

We record a simple fact about the existence of flow lines for the two concerned ancient solutions.

\begin{lemma}
    \label{lem:flow-line-existence}
    Let $\Sigma_t$ be a translating bowl soliton or a shrinking cylinder in $\bb R^{n+1}$, which is an ancient mean curvature flow.
    Then given any $\ovl t < 0$ and $ p\in \Sigma_{\ovl t},$ we can find $\gamma\colon (-\infty, \ovl t]\to \bb R^{n+1}$ such that it is a gradient flow line with $\gamma\pr{\ovl t} = p.$
\end{lemma}

When proving the uniqueness of limit flows along a gradient flow line, we will need a backward stability result of cylindricality for translating bowl solitons.

\begin{lemma}
\label{lem:backward-stability}
    Let $\Sigma_t$ be a translating bowl soliton and let $\gamma(t)$ be a gradient flow line.
    Given any $\varepsilon < \frac{1}{10}$, there exists $\delta \in (0, \varepsilon)$ and $\mathcal T<\infty$ such that if $(\gamma(0),0)$ is $2\delta$-cylindrical, then $(\gamma(t),t)$ is 
    $\frac{\varepsilon}{2}$-cylindrical for all $t\in \left[ - \mathcal T H^{-2} \pr{\gamma(0), 0}, 0\right]$ and $\frac{\delta}{2}$-cylindrical for all $t\in \left(-\infty, - \mathcal T H^{-2} \pr{\gamma(0), 0}\right].$
    Moreover, if $\gamma(0)$ is not the tip of the bowl soliton, then given any $\delta' > 0,$ there exists $\mathcal T' < \infty$ such that $(\gamma(t),t)$ is $\delta'$-cylindrical for all 
    $t\in \left(-\infty, - \mathcal T' H^{-2} \pr{\gamma(0), 0}\right].$
\end{lemma}

A general version of the backward improvement was done by Gianniotis--Haslhofer in {\cite[Proposition 4.2]{GH}}, where they dealt with all ancient non-collapsed mean curvature flows.
Here, we only need the result when the model is a bowl soliton or a cylinder, so we can get a stronger improvement statement.

We end this section by recording a convergence result of flow lines when looking at limit flows along $\gamma.$
This lemma follows from a general statement in Appendix~\ref{app:A} (Lemma~\ref{lem:flow-line-conv-general}).

\begin{lemma}\label{lem:flow-line-conv}
    Let $\gamma(t)$ be a gradient flow line of a mean curvature flow $M_t$
    and we write 
    $H(t) 
    := H\pr{\gamma(t),t}
    = H_t\pr{\gamma(t)}.$
    Given a sequence $t_i\to 0,$ if the sequence of flows
    $M^{t_i}_t
    := M^{\pr{\gamma(t_i), t_i, H(t_i)}}_t
    $
    converges to an ancient solution $M^\infty_t$ which is a translating bowl soliton or a shrinking cylinder, then the rescaled flow line
    \begin{align*}
        \gamma^{t_i}(t) :=
        H(t_i) \pr{\gamma\pr{H(t_i)^{-2} t + t_i} - \gamma(t_i)}
    \end{align*}
    converges to a flow line $\gamma^\infty(t)$ of $M^\infty_t$ ($t\in(-\infty, 0]$) smoothly on any compact time interval in $(-\infty,0].$
\end{lemma}

\section{\bf Mean convex flows with neck singularities}
\label{sec:AT-neck}

In this section, we fix a mean convex mean curvature flow $M_t$ ($t\in[-T, \infty)$)
and consider its arrival time function $u$ defined on $\Omega$ with $\bd\Omega = M_{-T}.$
We will then fix a gradient flow line $\gamma(t)$ and suppose that it has a neck singularity $(x_0, t_0)$ as a limit point and is parametrized by the time of the flow.
Without loss of generality, we may assume~$(x_0, t_0)=(0,0).$
All cylinders in this section are $1$-cylinders.

\subsection{\bf Uniqueness of limit points}
First, we prove the \L ojasiewicz theorem for $\gamma$.
This is based on the structure theorems of cylindrical singularities by Colding--Minicozzi \cite{CM162} (cf. \cite{G24}) and some classical results about limit points of $C^1$ curves.\footnote{In the theory of ordinary differential equations, the set of limit points of a gradient flow line is an interesting object. 
One can, for instance, see \cite[Section 3.2]{P01}, based on which the arguments about the connectedness in the proof are structured.}

\begin{thm}[\L ojasiewicz for neck singularities]
\label{thm-L}
    If $\gamma$ has a limit point which is a neck singularity, then this limit point is the unique limit point of $\gamma.$
\end{thm}

We remark that a possible alternative approach to proving the uniqueness of limit points is to appeal to the uniqueness of limit flows, which will be proved in Theorems~\ref{cylindrical-limit-flow} and~\ref{bowl-limit-flow}. 
See Remark~\ref{rmk:another-proof-of-L}. 
However, in contrast to the current version of Theorem \ref{bowl-limit-flow}, this alternative approach requires additional work to establish the uniqueness of the axes of bowl solitons, since the limit point may not be unique a priori.
Here, we instead opt for a direct proof, as the structure theory of singular sets developed by Colding--Minicozzi~\cite{CM162} is particularly effective in the presence of spherical and neck singularities. This strategy also takes advantage of the special structure of one-dimensional curves and avoids difficulties arising from non-uniqueness of limit points, thereby simplifying the preparation in Section~\ref{sec:unique-limit-flow} for the proof of the Arnold--Thom conjecture.

\begin{proof}
    We start with a general claim about the set of limit points of a gradient curve.

    \begin{claim}
        \label{claim:grad-curve}
        Let $K$ be the set of limit points of $\gamma(t)$ as $t\to 0.$
        That is,\footnote{One can also define $K$ by looking at the closure of the spacetime track of $\gamma$ intersected with the time-zero slice.
        With this definition, the closedness will follow directly.}
        \begin{align*}
            K:=\set{
            x\in\Omega:
            \lim_{i\to\infty}\gamma(t_i)=x
            \text{ for some }t_i\nearrow 0
            }.
        \end{align*}
    Then $K$ is a compact connected subset of $\Omega.$
    \end{claim}

    We prove the claim first.
    To show the compactness, it suffices to show that $K$ is closed since $K$ is a bounded region.
    The closedness then follows from a standard diagonal argument.
    In fact, if $x\in \ovl K,$ then we can find $x^j\in K$ with $x^j\to x$ as $j\to\infty.$
    For each $j,$ we can find $t^j_i\in [-T,0)$ such that $\gamma(t^j_i)\to x^j$ as $i\to\infty.$
    We may assume for each $j,$ $t^j_i$ is increasing.
    Then for each $j,$ we take $i(j)\in\bb N$ such that 
    $$d_{\rm euc}\pr{\gamma \pr{t^j_{i(j)}}, x^j}\le \frac 1j.$$
    The triangle inequality then implies $\gamma \pr{t^j_{i(j)}}\to x$ as $j\to\infty.$
    This proves that $K=\ovl K.$
    Thus, $K$ is closed and hence compact.

    Next, we show the connectedness.
    If $K$ were not connected, we could write $K=U\cup V$ where $U$ and $V$ are disjoint proper closed subsets of $K$ (and hence $\bb R^{n+1}).$
    Let $\delta > 0$ be the distance between $U$ and $V.$ 
    We can then find two sequences $t_i\to 0$ and $s_i\to 0$ such that 
    $$d_{\rm euc}\pr{\gamma(t_i), U}
    < \delta/2 
    < d_{\rm euc}\pr{\gamma(s_i), U}$$ 
    for all $i.$ 
    By the continuity of the function $d_{\rm euc}\pr{\gamma(\cdot), U},$ we are able to find $\tau_i\to 0$ such that $d_{\rm euc}\pr{\gamma(\tau_i), U}=\delta/2$ for all $i\in\bb N.$
    Passing to a subsequence, we get a limit $\lim\limits_{i\to\infty}\gamma(\tau_i)\in K,$ which cannot be in either $U$ or $V,$ a contradiction.
    Thus, $K$ is connected, and Claim \ref{claim:grad-curve} is proven.

    Now, we prove the conclusion by contradiction.
    Suppose $\# K>1$, i.e., limit points are not unique. 
    Let $p\in K$ be a neck singularity.
    The structure theorem in \cite{CM162} then implies that $K\cap \ovl B_\delta(p)$ consists of only spherical and neck singularities and is contained in a $C^1$ curve when $\delta>0$ is small.
    Since $K$ is a compact connected set, after shrinking $\delta$ if necessary, $K\cap \ovl B_\delta(p)$ is also compact and connected and hence itself forms a $C^1$ curve (instead of being a proper subset of it) if $\# \pr{K\cap\ovl B_\delta(p)}>1.$
    Since the arrival time function is $C^2$ near a spherical singularity by Huisken's result \cite{H93}, every point in $K\cap \ovl B_\delta(p)$ must be a neck singularity.
    If this $C^1$ curve had an interior point, the arrival time function would be $C^2$ near this point by Theorem~\ref{thm:C2}, and the \L ojasiewicz inequality \cite[Theorem~2.1]{CM19} implies the uniqueness of limit points, that is, $\# K = 1,$ a contradiction.
    Hence, $K\cap \ovl B_\delta(p)$ consists of one point,
    so $K$ only contains isolated points.
    Since $K$ is connected, the only possibility is $\# K=1.$
    This again contradicts the assumption that limit points are not unique. 
    Therefore, the conclusion follows.
\end{proof}

By Theorem~\ref{thm-L}, every gradient flow line having a neck singularity as a limit point is actually convergent to it.
This implies the first part of Theorem~\ref{thm:main}.
Thus, we can freely use Proposition~\ref{prop:convergence} in the rest of this section.

\subsection{Uniqueness of limit flows}
\label{sec:unique-limit-flow}

In the rest of this section, we only consider limit flows obtained from blowing up along a sequence of the form $(x_i, t_i) = \pr{\gamma(t_i), t_i}.$
Thus, for simplicity, given $t_0<0,$ we let
\begin{align*}
    M^{t_0}_t
	:= M^{\pr{\gamma\pr{t_0}, t_0, H\pr{t_0}}}_t
    = H(t_0) \pr{M_{H(t_0)^{-2} t + t_0} - \gamma(t_0)}
\end{align*}
be the rescaled flow we care about,
where $H(t_0) := H\pr{\gamma(t_0), t_0} = H_{t_0}\pr{\gamma(t_0)}.$
A limit flow obtained from such a rescaled sequence will be called a limit flow along $\gamma.$

We start to look at the two types of limit flows along $\gamma.$
By Proposition~\ref{prop:convergence}, they are either shrinking cylinders or translating bowl solitons.
As mentioned, a mean convex mean curvature flow $M_t$ ($t\in[-T, \infty)$) is fixed.
First, we show a backward improvement result for cylindricality.

\begin{lemma}
    \label{backward-cylindrical-propogation}
    For any $\varepsilon \in \pr{0,\frac{1}{10n}}$, there exist $\delta\in (0,\varepsilon)$ and $\tau < 0$ such that if $\pr{\gamma(\ovl{t}),\ovl{t}}$ is $\delta$-cylindrical 
    for some $\ovl{t} \in (\tau,0)$, then $(\gamma(t),t)$ is $\varepsilon$-cylindrical 
    for all  $t\in[\tau, \ovl{t}].$
\end{lemma}

We remark that the constants $\delta$ and $\tau$ depend on $M_t,$ and we can apply the lemma to other mean convex flows.

\begin{proof}
Given $\varepsilon\in \pr{0,\frac{1}{10n}},$ we take $\delta\in (0,\varepsilon)$ and $\mathcal{T}>0$ to be the constants from Lemma~\ref{lem:backward-stability}.
We claim that there exists $\tau < 0$ such that if $\ovl t\in[\tau, 0)$ and $\pr{\gamma(\ovl{t}),\ovl{t}}$ is $\delta$-cylindrical, then $(\gamma(t), t)$ is
\begin{align}\label{first-claim}
    \varepsilon\text{-cylindrical for }
    t\in \left[\ovl{t}-H\pr{\ovl{t}}^{-2}\mathcal{T},\ovl{t}\right],
    \text{ and }
    \delta\text{-cylindrical for }
    t  = \ovl{t} - H(\ovl{t})^{-2}\mathcal{T}.
\end{align}
If we were not able to find such a $\tau < 0$, then we could find $t_i \to 0$ such that
$\pr{\gamma(t_i),t_i}$ is $\delta$-cylindrical but either $(\gamma(t),t)$ is not $\varepsilon$-cylindrical for some $t \in \left[t_i - H(t_i)^{-2} \mathcal{T}, t_i\right]$ 
or $(\gamma(t), t)$ is not $\delta$-cylindrical at $t = t_i - H(t_i)^{-2} \mathcal{T}.$
By Proposition~\ref{prop:convergence}, we can pass to a subsequence such that
\begin{align}\label{Convergence-1}
    M_t^{t_i}\rightarrow M^{\infty}_t
\end{align}
where $M_t^{\infty}$ is a translating bowl soliton or a shrinking cylinder. 
By the assumption, $(0,0)$ is $2\delta$-cylindrical in $M_t^\infty$.  
By Lemma~\ref{lem:flow-line-existence}, there is a gradient flow line $\gamma^{\infty}(t)$ of $M_t^\infty$ for $t \in (-\infty,0]$ with $\gamma^\infty(0) = 0,$ and Lemma~\ref{lem:flow-line-conv} implies that the rescaled flow line
\begin{align}\label{gamma-conv-i}
    H(t_i) \pr{\gamma(H(t_i)^{-2} t + t_i)-\gamma(t_i)}
    \to \gamma^\infty(t)
\end{align}
smoothly on any compact time interval in $(-\infty,0].$
By Lemma~\ref{lem:backward-stability}, $\pr{\gamma^\infty(t), t}$ is $\frac{\delta}2$-cylindrical at $t = -\mathcal{T}$ and is $\frac{\varepsilon}2$-cylindrical when $t \in [-\mathcal{T},0].$ 
Since $\gamma^\infty\pr{[-\mathcal T, 0]}$ is contained in a compact ball in $\bb R^{n+1},$ the convergence \eqref{Convergence-1} and \eqref{gamma-conv-i} implies that for all sufficiently large $i$, $(\gamma(t),t)$  is $\varepsilon$-cylindrical when $t \in \left[t_i -H(t_i)^{-2}\mathcal{T}, t_i\right]$ and is $\delta$-cylindrical at $t = t_i - H(t_i)^{-2}\mathcal{T}$. 
This leads to a contradiction, and the claim \eqref{first-claim} is proven.

We fix the constants $\tau,$ $\delta,$ and $\mathcal{T}$ from the claim above.  
Let $\ovl{t}_1 = \ovl{t}$. 
If $\tau \leq \ovl{t}_1 - H(\ovl{t}_1)^{-2} \mathcal{T}$,  
we pick $\ovl{t}_2 = \ovl{t}_1 - H(\ovl{t}_1)^{-2}\mathcal{T}$, 
otherwise we stop. 
We can repeat the process and pick 
$0  > \ovl{t}_1 > \ovl{t}_2 > \cdots$ 
where 
$\ovl{t}_{i+1} = \ovl{t}_i - H(\ovl{t}_i)^{-2} \mathcal{T}$ 
and we stop at $\ovl{t}_\ell$ if 
$\ovl{t}_\ell \geq \tau$ 
but 
$\ovl{t}_\ell - H\pr{\ovl{t}_\ell}^{-2} \mathcal T <\tau$.

By the claim \eqref{first-claim}, $(\gamma(t),t)$ is $\varepsilon$-cylindrical when $t \in \left[\ovl{t}_{i+1},\ovl{t}_i\right]$ and $(\gamma(\ovl{t}_i), \ovl{t}_i)$ is $\delta$-cylindrical.
Therefore,
\begin{align*}
    \partial_t H = \Delta H + |A|^2H \geq \left(\frac{1}{n-1}-C\varepsilon\right)H^3 > 0,
\end{align*}
noting that $\Delta H_{\rm cyl} = 0$ and $|A_{\rm cyl}|^2H_{\rm cyl} = \frac{1}{n-1}H_{\rm cyl}^3$ on an exact cylinder $S^{n-1}\times\mathbb{R}$. 
Consequently, $H(t) = H(\gamma(t),t)$ is strictly increasing on~$\left[\ovl{t}_i, \ovl{t}_1\right].$
This means that $\ovl{t}_j < \ovl{t}_1 - (j-1) H\pr{\ovl{t}}^{-2} \mathcal{T}.$
Therefore, the process must stop at $\ovl{t}_\ell$ for some finite $\ell$. 
Now, we use the claim again to conclude that $(\gamma(t),t)$ is $\varepsilon$-cylindrical when $t \in \left[\ovl{t}_l - H\pr{\ovl{t}_l}^{-2} \mathcal{T}, \ovl{t}\right]$. 
Since $\ovl{t}_l - H\pr{\ovl{t}_l}^{-2} \mathcal{T} \leq \tau,$ the assertion of the lemma follows.
\end{proof}

The claim in the proof asserts that when $t$ is sufficiently close to 0, being cylindrical at a point implies being cylindrical when going backwards along the gradient flow line for a period of time comparable to the square of the cylindrical radius. 
The upshot is that the cylindrical property actually improves when the flow goes back for a while. 
Therefore, in the proof, we are able to iterate the process and show that the cylindricality is preserved when going backwards along the flow line all the way up to a fixed negative time.
Using the backward improvement, we can show the uniqueness of cylindrical limit flows along the flow line $\gamma.$

\begin{thm}\label{cylindrical-limit-flow}
    If the sequence $M^{\td t_i}_t$ converges to  a shrinking cylinder along one sequence $\td{t}_i \rightarrow 0$, then $M^{t_i}_t$ converges to a shrinking cylinder along any sequence $t_i \rightarrow 0$. 
\end{thm}

We remark that this limit cylinder is different from the cylinder given by the tangent flow.
Recall that we translate the singularity such that the tangent flow is centered at the origin,
but the limit flow is not centered at the origin (since it passes through the origin regularly), though the two cylinders are parallel.
Once the uniqueness of limit tangents is obtained (Theorem~\ref{thm:AT-neck}), the limit tangent will tell us how these two cylinders are different from each other.

\begin{proof}
Take an arbitrary sequence $t_i \to 0.$ 
By Proposition~\ref{prop:convergence}, the sequence $M^{t_i}_t$ converges to a non-collapsed ancient solution $M^\infty_t.$
The theorem follows if we can show that $M_t^\infty$ is a shrinking cylinder.

For any $\varepsilon > 0$, we take the constants $\delta$ and  $\tau$ from Lemma~\ref{backward-cylindrical-propogation}. 
By the convergence, $\pr{\gamma\pr{\td{t}_i}, \td{t}_i}$ is $\delta$-cylindrical for all sufficiently large $i,$ and Lemma~\ref{backward-cylindrical-propogation} forces $\pr{\gamma(t),t}$ to be $\varepsilon$-cylindrical when $t\in \left[\tau, \td{t}_i\right).$
Letting $i\rightarrow \infty,$ we obtain that $\pr{\gamma(t),t}$ is $\varepsilon$-cylindrical when $t\in [\tau, 0).$
We can now conclude that $M^\infty_t$ is $\varepsilon$-cylindrical at $(0,0)$ for any $\varepsilon > 0.$
Sending $\varepsilon\to 0$ proves the assertion.
\end{proof}

Next, we deal with the case of translating bowl solitons.
We will be able to show that bowl-type limit flows along a gradient flow line are also unique, even with the same tip point.
First, we show that when a translating bowl soliton appears as a limit flow, the singularity must be the tip of the bowl.

\begin{lemma}\label{converges-to-tip-of-bowl}
     If the sequence $M^{t_i}_t$ converges to
     a translating bowl soliton along a sequence $t_i\rightarrow 0$, then the origin is the tip of the limit bowl soliton. 
     Moreover, the rescaled flow lines converge to the flow line of the tip of the limit bowl soliton, which moves at constant speed on a straight line and arrives the origin at time $0.$
\end{lemma}

\begin{proof}
    Assume that $M_t^{t_i}$ converges to $\Sigma_t = \Sigma + t v$, a translating bowl soliton that passes through the origin $0$ at time $0.$ 
    It suffices to prove that $0$ is the tip of $\Sigma_0.$
    The second statement follows from Lemma~\ref{lem:flow-line-conv} and the fact that the flow line of the tip of a bowl soliton moves at a constant speed along a line. 
    
    Suppose toward a contradiction that $0$ is not the tip of $\Sigma_0$. 
    By Lemma~\ref{lem:flow-line-existence}, we can consider the flow line $\ovl{\gamma}(t)$ of $\Sigma_t$ with $\ovl{\gamma}(0)=0$.
    We fix any other sequence $\ovl{t}_i\rightarrow 0$. 
    By Proposition~\ref{prop:convergence}, $M^{\ovl t_i}_t$ converges to a non-collapsed ancient solution $M^{\infty}_t$ which is a translating bowl soliton or a shrinking cylinder. 
    
    Next, we fix an arbitrary $\varepsilon > 0$ and take the constants $\delta$ and  $\tau$ from  Lemma~\ref{backward-cylindrical-propogation}
    applied to~$M^\infty_t$.
    Since $0$ is not the tip,
    by Lemma~\ref{lem:backward-stability}, there exists $\mathcal{T} <\infty$ such that $(\ovl{\gamma}(t),t)$ is $\frac{\delta}{2}$-cylindrical when $t \leq -\mathcal{T}.$  
     By Lemma~\ref{lem:flow-line-conv},
     $\gamma^{t_i}(t)$ converges uniformly to $\ovl{\gamma}(t)$ when 
     $t \in [-2\mathcal{T},0].$
     Rescaling back implies that 
     $\pr{\gamma(t_i - H(t_i)^{-2} \mathcal{T}), t_i-H(t_i)^{-2}\mathcal{T}}$ 
     is $\delta$-cylindrical for all sufficiently large $i$.
     By Lemma~\ref{backward-cylindrical-propogation}, $(\gamma(t),t)$ is $\varepsilon$-cylindrical when $t\in [\tau,t_i-H(t_i)^{-2}\mathcal{T}).$ 
     Since $t_i$ and $H(t_i)^{-1}$ both converge to $0,$ this implies that $(\gamma(t),t)$ is $\varepsilon$-cylindrical when $t\in [\tau,0).$ 
     In particular, this means that $M^{\infty}_t$ is $2\varepsilon$-cylindrical at $(0,0).$
     Since $\varepsilon$ is arbitrary, we conclude that $M^{\infty}_t$ is a shrinking cylinder. 
     
     To summarize, we find a sequence $\ovl{t}_i$ such that $M^{\ovl t_i}_t$ subsequentially converges to a cylinder.  
     Theorem~\ref{cylindrical-limit-flow} then forces $M^{t_i}_t$ to converge to a shrinking cylinder as well, but that contradicts our assumption that $M^{t_i}_t$ converges to a bowl soliton.
\end{proof}

We conclude the section with the uniqueness result for translating limit flows as follows.

\begin{thm}\label{bowl-limit-flow}
    If the sequence $M^{\td t_i}_t$  converges to 
     a bowl soliton along a sequence $\td{t}_i\rightarrow 0$, then $M^{t_i}_t$ converges to the same bowl soliton along any sequence $t_i\rightarrow 0$.
\end{thm}
\begin{proof}
    By Lemma~\ref{converges-to-tip-of-bowl}, after translations and rotations, the limit of $M^{\td t_i}_t$ along the sequence
    $\td{t}_i$ is the bowl soliton $\Sigma_t = \Sigma + t e_0$ where $0$ is the tip of $\Sigma$ and $e_0$ is the unit vector pointing in the direction of the $x_0$-axis in $\bb R^{n+1}$ (unique up to a sign).
    Given any other sequence $t_i\rightarrow 0,$ we assume that $M^{t_i}_t$ converges to $M_t^{\infty}$ along the sequence $t_i\rightarrow 0$. 
    By Theorem~\ref{cylindrical-limit-flow}, $M_t^{\infty}$ can't be a shrinking cylinder. 
    Hence, by Proposition~\ref{prop:convergence}, the limit $M_t^{\infty}$ must be a bowl soliton with its axis parallel to that of $\Sigma$.
    By Lemma~\ref{converges-to-tip-of-bowl}, $0$ must be the tip of $M_0^{\infty}$. 
    Hence, $M_t^{\infty}$ is either $\Sigma + te_0$ or $-\Sigma - te_0$, where $-\Sigma$ is the reflection of $\Sigma$ along $e_0$ direction.
    We claim that the latter situation can not happen.

    Suppose toward a contradiction that $M_t^{\infty} = -\Sigma - t e_0$. 
    Then the (outer) unit normal of $M_0^{\infty}$ at the origin is $e_0.$
    Therefore, if we let $\N(t)$ denote the unit normal vector at $\gamma(t)\in M_t,$ the contradiction assumption implies that $\N( t_i)\rightarrow e_0.$
    Similarly, we have  $\N\pr{\td{t}_i} \rightarrow -e_0$.  
    In particular, $\langle\N(t_i),e_0\rangle > 0$ and $\pair{\N\pr{\td{t}_i},e_0} < 0$ for sufficiently large $i$. 
    Since $\N(t)$ is continuous, for large $i,$ we can find $s_i$ between $t_i$ and $\td{t}_i$ such that $\langle\N(s_i),e_0\rangle=0$. 
    This construction implies $s_i\rightarrow0$. 

    By Proposition~\ref{prop:convergence}, the sequence $M_t^{s_i}$ subsequentially converges to a non-collapsed limit $\td{M}_t^{\infty}$ that is either a bowl soliton or a cylinder with axis parallel to $e_0$. 
    Moreover, at the origin, $\N\perp e_0$ in  $\td{M}_0^{\infty}$ by our choice of $s_i$'s. 
    This forces $\td{M}_t^{\infty}$ to be a cylinder. 
    However, Theorem~\ref{cylindrical-limit-flow} then forces the limit of $M_t^{\td t_i}$ to be a cylinder. 
    This is a contradiction.
    That is, we prove that the limit of $M_t^{t_i}$ must coincide with $\Sigma_t$. 
    Since $t_i$ is arbitrary, the theorem is proven.
\end{proof}


\subsection{Uniqueness of limit tangents}
\label{sec:proof-main-theorems}

Based on Theorem~\ref{cylindrical-limit-flow} and Theorem~\ref{bowl-limit-flow}, we may call a flow line cylinder-type or bowl-type depending on the limit flows along it.
Based on the types of limit flows along a gradient curve limiting to a neck singularity, there are two possibilities.
That is, it is either bowl-type or cylinder-type.
For the former case, we will deal with it directly;
for the later case, we will reduce it to the work of Colding--Minicozzi \cite{CM18, CM19}.
Since we will work in a similar situation in Section~\ref{sec:4.3}, we single out this result for cylinder-type singularities for all types of generalized cylinders.

\begin{lemma}
    \label{lem:cylinder-type-case}
    Let $M_t$ be a closed mean convex mean curvature flow in $\bb R^{n+1}$ and let $u$ be its arrival time function with $u(0)=0$.
    Suppose $0$ is a cylindrical singularity whose tangent flow is a $k$-cylinder 
    and suppose $\gamma$ is a gradient flow line for $u$ converging to $0.$
    If all limit flows along $\gamma$ are $k$-cylinders with axes parallel to that of the tangent flow, then the unit tangents of $\gamma$ converge.
\end{lemma}

This is a consequence of the geometry of shrinking cylinders and theories established by Colding--Minicozzi.
We will apply Lemma~\ref{lem:cylinder-type-case} with $k=1$ in this section and with general $k$ in Section~\ref{sec:4.3}.
We remark again that Lemma~\ref{lem:cylinder-type-case} already assumes the convergence of the flow line, which is fine since we always prove the \L ojasiewicz theorem first before applying this lemma both here and in Section~\ref{sec:4.3}.
In Section~\ref{sec:gen-cyl}, we make use of similar techniques in the proof of Theorem~\ref{thm-L2}, while different aspects of the evolution inequalities are explored.
See Remark~\ref{rmk:sec34-cyl}.

\begin{proof}
[Proof of Lemma~\ref{lem:cylinder-type-case}]
After a rotation, we may assume the axis of the tangent flow is $\set{0}\times \bb R^k$ where $0=0^{n-k+1}\in\bb R^{n-k+1}.$
Let $\Pi_{\rm axis}$ be the orthogonal projection onto $\set 0 \times \bb R^k$ and define its complement $\Pi:= {\rm Id}_{\bb R^{n+1}} - \Pi_{\rm axis}.$

The assumption implies that $ \gamma' = \h$ behaves similarly to that of the shrinking cylinder near the limit point. 
More precisely, there exists $\tau < 0$ such that for all $t \in [\tau,0)$ we have \\
(1) $H = H(t) =  H\pr{\gamma(t),t}> 0,$\\
(2) $\frac{1}{2n}H^3 \leq\partial_t H \leq 2H^3$, and\\
(3) $\abs{\Pi_{\rm axis}\circ \gamma'(t)} 
= |\langle \gamma',e_0\rangle(t)| 
= o\pr{|\gamma'|}$ as $t\to 0.$\\
The second inequality implies that $-4\leq \frac{d}{dt}H^{-2} \leq -\frac{1}{n}$.
Since $H(t)$ is smooth on $[\tau,0)$ and $H(t)\rightarrow \infty$ as $t\rightarrow 0$, we integrate the differential inequality and obtain that
\begin{align*}
	\sqrt{\frac{1}{-4t}} \leq H(t)\leq \sqrt{\frac{n}{-t}}.
\end{align*}
Therefore,
\begin{align}\label{main;gamma-distance-up-bound}
	|\gamma(t)|\leq \int_t^0 |\gamma'(s)|ds = \int_t^0 |H(s)|ds \leq 2\sqrt{-nt}.
\end{align}
In particular, $\gamma(t) \in B_{2n\sqrt{-t}}(0)$ for all $t\in [\tau,0)$.

Next, we compute
\begin{align}\label{main;gamma-axis-proj-small}
	\abs{\Pi_{\rm axis}\circ \gamma(t)} \leq \int_t^0 \abs{\Pi_{\rm axis}\circ \gamma'(s)} ds = o(\sqrt{-t})
\end{align}
as $t\to 0.$
The geometry of the tangent flow and (\ref{main;gamma-distance-up-bound})implies that 
\begin{align}\label{main;gamma-dist-low-bound}
    \abs{\Pi \circ\gamma(t)}  \geq \sqrt{-t}.
\end{align} 
Combining \eqref{main;gamma-axis-proj-small} and \eqref{main;gamma-dist-low-bound}, we see that there exists $d<\infty$ large enough such that 
\begin{align}\label{gamma(t)-in-Cd}
    \gamma(t)\in
    C_d:= \set{x\in \Omega: |\Pi_{\rm axis}(x)| \le d|\Pi(x)|}
\end{align}
for $t\in[\tau,0).$

Now, we are going to apply the theories established by Colding--Minicozzi in \cite{CM18, CM19}.
First, by~\cite[Theorem 1.3]{CM18}, there exists $\delta>0$ such that $u$ is $C^2$ and has the unique critical point $0$ in $C_d\cap B_\delta.$
The proofs of \cite[Theorems~2.1, 2.2, 3.3]{CM19} then go through in the region $C_d\cap B_\delta$ since the key property is the $C^2$-regularity of $u.$
To be more precise, the proof of \cite[Theorem~2.1]{CM19}, which is essentially based on the fundamental theorem of calculus, implies that given any $\ovl\varepsilon>0,$ we can find $\ovl\delta\in(0,\delta)$ such that
\begin{align*}
    \abs{\n u(q) + \frac{q-p}{n-k}}
    \le \ovl \varepsilon |q-p|
\end{align*}
for all $p,q\in C_d\cap B_{\ovl \delta}.$
Since this is true for any $\ovl\varepsilon>0,$ in particular, the inequality implies
\begin{align}\label{CM19-(2.2)}
    \frac{|\n u|^2}{-u}(p)
    \to \frac 2{n-k}
\end{align}
as $p\in C_d\cap B_\delta$ tends to $0.$
The relation \eqref{CM19-(2.2)} and the property of being $C^2$ are the only two ingredients in the proofs of \cite[Theorems~2.2 and 3.3]{CM19}.
Thus, \cite[Theorem~3.3]{CM19} is valid in our case by \eqref{gamma(t)-in-Cd} and \cite[Theorem 1.3]{CM18}.

With \cite[Theorem~3.3]{CM19},
we know that if the summability estimate \cite[Equation (3.17)]{CM19} is true for the rescaled mean curvature flow obtained by rescaling with respect to a cylindrical singularity $(0,0),$ then Theorem~\ref{thm:AT-neck} follows.
Thus, the summability estimate \cite[Equation~(3.17)]{CM19} is the only missing piece.
However, in \cite[Sections 4, 5, and 6]{CM19}, Colding--Minicozzi proved \cite[Equation~(3.17)]{CM19} for any rescaled mean curvature flow converging to a generalized cylinder without additional conditions (like $C^2$-regularity of $u$).
In particular, the estimate \cite[Equation~(3.17)]{CM19} is true for the rescaled mean curvature flow obtained with respect to the cylindrical singularity $(0,0)$ in our case.
This proves Lemma~\ref{lem:cylinder-type-case}.
\end{proof}

Now, we use the uniqueness of limit flows and Lemma~\ref{lem:cylinder-type-case} to prove the main theorem of this section.
It immediately implies Theorem~\ref{thm:main} and the case of two-convex flows claimed in Section~\ref{sec:main-results} since the arrival time function is $C^2$ at spherical points and hence Colding--Minicozzi's work already covers this case.

\begin{thm}[Arnold--Thom for neck singularities]\label{thm:AT-neck}
    Let $M_t$ be a closed mean convex mean curvature flow in $\bb R^{n+1}$ and let $u$ be its arrival time function.
    If $\gamma$ is a gradient flow line for $u$ with a neck limit point,
    then the limit point of $\gamma$ is unique and the unit tangents of $\gamma$ converge.
\end{thm}

\begin{proof}
    The uniqueness of limit points is addressed in Theorem \ref{thm-L}.
    For the uniqueness of limit tangents, we may assume the limit point is the origin and that it happens at time $t=0$. 
    In particular, $\lim\limits_{t\rightarrow 0}\gamma(t)=0$.

    There are two cases based on the types of limit flows along $\gamma.$

    \noindent\textbf{Case 1:} A limit flow along $\gamma$ is a translating bowl soliton.

	By Theorem~\ref{bowl-limit-flow}, there is a unique bowl soliton $\Sigma$ such that every limit flow along $\gamma$ is $\Sigma_t = \Sigma + te_0$ where $e_0$ is a unit vector parallel to the axis of the cylindrical tangent flow.
	The convergence implies that given any $t_i\to 0,$ after passing to a subsequence, we have $\lim\limits_{i\to \infty} \frac{\gamma'(t_i)}{\abs{\gamma'(t_i)}} = e_0.$
	This implies the genuine convergence $\lim\limits_{t\to 0^-} \frac{\gamma'(t)}{\abs{\gamma'(t)}} = e_0.$
	
	\noindent\textbf{Case 2:} A limit flow along $\gamma$ is a shrinking cylinder.

By Theorem~\ref{cylindrical-limit-flow}, given any $t_i\to 0,$ after passing to a subsequence, the limit flow is a shrinking cylinder parallel to the axis of the neck tangent flow.
Thus, applying Lemma~\ref{lem:cylinder-type-case} finishes the proof of this case.
This completes the proof of Theorem~\ref{thm:AT-neck}.
\end{proof}


\section{\bf Mean convex flows with non-degenerate cylindrical singularities}
\label{sec:gen-cyl}

\subsection{Uniqueness of limit flows}
In this section, we will study mean convex flows with singularities modeled on $k$-cylinders for general $k\ge 1.$
We are going to define non-degenerate cylindrical singularities.
We will need some good properties of limit flows at such singularities.
To this end, we introduce the notion of Gaussian densities.

\begin{definition}
    Let $M_t^n$ be a smooth mean curvature flow in $\bb R^{n+1}$ and $(x_0, t_0)\in \bb R^{n+1}\times \bb R.$ 
    The Gaussian density of $M_t$ at $(x_0, t_0)$ with scale $r$ is defined by
    \begin{align*}
        \Theta_{M_t}(x_0, t_0,r)
        := \frac 1{\pr{4\pi r^2}^{n/2}}
        \int_{M_{t_0-r^2}} e^{-\frac{|x-x_0|^2}{4r^2}} dx
    \end{align*}
    and the Gaussian density of $M_t$ at $(x_0, t_0)$ is the limit
    \begin{align*}
        \Theta_{M_t}(x_0, t_0)
        := \lim_{r\to 0^+} \Theta_{M_t}(x_0, t_0,r).
    \end{align*}
\end{definition}

A closely related notion is the {\rm entropy} introduced by Colding--Minicozzi \cite{CM12}.
For an $n$-dimensional hypersurface $\Sigma\sbst\bb R^{n+1},$ its entropy is defined by
\begin{align*}
    \lambda(\Sigma)
    := \sup_{\substack{x_0\in\bb R^{n+1}\\ t_0>0}}
    \frac 1{(4\pi)^{n/2}}
    \int_{t_0\Sigma + x_0}
    e^{-\frac{|x|^2}{4}}
    dx.
\end{align*}
By Huisken's monotonicity \cite{H90}, given a mean convex flow\footnote{
A mean convex level set flow supports a Brakke flow, for which Huisken's monotonicity applies.
} 
$M_t$ and a point $(x_0, t_0),$ the density~$\Theta_{M_t}(x_0, t_0, \cdot)$ and the entropy $\lambda\pr{M_{(\cdot)}}$ are always monotone.
Thus, the definitions are both valid as long as we start from an initial hypersurface with finite entropy.
Since we always start from a smooth closed hypersurface, in particular, they are always well-defined.

By definition, we have $\lambda(\Sigma) = \lambda\pr{\Sigma\times\bb R^k}$ for any hypersurface $\Sigma$ and $k\ge 0.$
Thus, when talking about the entropy of a cylinder, we will always look at its spherical factor.
We record a simple lemma about good properties of the entropy functional.
The second one is based on the calculations done by Stone \cite{S94}.

\begin{lemma}\label{lem:ent}
    The entropy functional satisfies the following properties.\\
    (1) (\cite{CM12}) It is lower semi-continuous in the $C^\infty_{\rm loc}$-topology.\\    
    (2) (\cite{S94}) The entropies of all generalized cylinders satisfy 
    \begin{align*}
        \lambda\pr{S^1}
    > \lambda\pr{S^2}
    > \cdots
    > \lambda\pr{S^n}
    >\sqrt 2
    > \lambda\pr{\bb R^n}
    = 1
    \end{align*}
    for any $n\in\bb N.$
\end{lemma}

For general dimensions, we follow \cite{SS} and consider non-degenerate singularities in the following way.

\begin{definition}\label{def:nondeg}
    Let $M_t$ be a mean curvature flow and $(x_0, t_0)$ be a cylindrical singularity.
    We call $(x_0, t_0)$ a non-degenerate cylindrical singularity if all non-trivial special limit flows are multiplicity one cylinders.
\end{definition}

Based on the definition, a priori, it is possible that at a non-degenerate cylindrical point, different limit flows are different types of cylinders (with axes of different dimensions and directions).
In the following theorem, we prove that this cannot happen at a non-degenerate cylindrical point.

\begin{thm}\label{lem:nondeg-unique}
    Let $M_t$ be a mean curvature flow with finite entropy and $(x_0, t_0)$ is a singularity of the flow.
    If $(x_0, t_0)$ is a non-degenerate cylindrical singularity whose tangent flow is given by a $k$-cylinder, then all non-trivial special limit flows are $k$-cylinders with axes parallel to that of the tangent flow.
\end{thm}

This leads to the first part of Theorem~\ref{thm:nondeg-cyl}.
We will use this to prove the Arnold--Thom conjecture for non-degenerated cylindrical singularities later.

\begin{proof}
    We will first show that all limit flows are the same type of cylinders.
    Then we will show that their axes all align.

    \noindent\textbf{Part 1:} Uniqueness of limit types.
    
    After translations and rotations, by the uniqueness of tangent flows \cite{CM15}, we may assume that the singularity is at $(0,0)$ and that the tangent flow at $(0,0)$ is 
    \begin{align*}
        \Gamma_t = \left\{(x_1,...,x_{n+1})\in\bb R^{n+1} \ \Big| \  
        \sum_{i=1}^{n-k+1}x_i^2 
        = -2(n-k)t\right\}
    \end{align*}
    for some $k\ge 0.$ 
    Note that $\Gamma_{-1}\in \C_k^0.$
    
    Let $(x_i, t_i)\rightarrow (0,0)$ be space-time points on $M_t$ and 
    $M^i_t = \lambda_i\pr{M_{\lambda_i^{-2}t + t_i} -x_i}$, where $\lambda_i\rightarrow \infty$. 
    By Brakke's compactness theorem \cite{B78}, after passing to a subsequence, $M^i_t$ converges to an integral Brakke flow $M_t^{\infty}$ ($t<0$). 
    By the non-degenerate assumption, such a limit must be a cylinder in $\C_\ell$ for some $\ell\in \{0,1,...,n\}$.\footnote{Note that this limit flow may not be centered at the origin.
    Moreover this limit flow is smooth for $t<0$, but may not be smooth at $t=0$.}
 
    By Brakke's regularity theorem \cite{B78} (cf. \cite{W05}), the convergence is smooth on compact sets of $\mathbb{R}^{n+1}\times (-\infty,0)$.
    By the lower semi-continuity given in Lemma~\ref{lem:ent}, the entropy of any limit flow is no more than the entropy of the tangent flow, which is a $k$-cylinder.
    Therefore, $\ell \le k$ or $\ell = n$ by Lemma \ref{lem:ent}.
    We aim to show that $\ell = k$ or $\ell = n$.

    We remark that the argument below will also give us a contradiction if $\ell > k.$ 
    Therefore, the entropy argument above is not completely necessary.

    Suppose toward a contradiction that this does not hold. 
    Then there would exist a limit flow $M_t^{\infty}$ as described above and equal to an $\ell$-cylinder with $\ell < k.$

    \begin{claim}\label{claim}
        For any $\varepsilon >0$, there exists $r_0 = r_0(\varepsilon)$ with the following property. 
        For any fixed $0<r_1<r_2\leq r_0$, $\frac{1}{\sqrt{-t}}M_t^{(x_i,t_i,r^{-1})}$ is $\varepsilon$-close\footnote{In the rest of Section~\ref{sec:gen-cyl}, we say $M$ is $\varepsilon$-close to $N$ if $M$ can be written as the graph of a function over $N$ in $C^{\lceil 1/\varepsilon \rceil}\pr{B_{1/\varepsilon}}$ with norm at most $\varepsilon.$} 
        to $\Gamma_{-1}$ for all sufficiently large $i$  whenever $t\in [-2,-\frac{1}{2}]$ and $r \in [r_1,r_2]$.
    \end{claim}
    
    \begin{proof}[Proof of Claim~\ref{claim}]
    We rewrite $M_t^{\pr{x_i,t_i,r^{-1}}}$ as
    \begin{align}\label{Shifting-the-flow-1}
        M_t^{\pr{x_i,t_i,r^{-1}}}
        = r^{-1} \pr{M_{t\pr{r^2 + \frac{t_i}{t}}} - x_i}
        =& \sqrt{1+\frac{t_i}{r^2t}} \cdot
        \sqrt{r^2 + \frac{t_i}{t}}^{-1}
        \pr{M_{t\pr{r^2 + \frac{t_i}{t}}} - x_i} \\
        =& \sqrt{1+\frac{t_i}{r^2t}} \left(M_t^{\pr{0,0,\sqrt{r^2 + t_i/t}^{-1}}} - x_i\right). \nonumber
    \end{align}
    Since the tangent flow at $(0,0)$ is $\Gamma_t$, for any $\varepsilon > 0$ we can find $r_0 = r_0(\varepsilon)$ such that $\frac{1}{\sqrt{-t}}M^{(0,0,r^{-1})}_t$ is $\frac{\varepsilon}{2}$-close to $\Gamma_{-1}$ for all sufficiently large $i$ when $t\in [-4,-\frac{1}{4}]$ and $r \in (0, 2r_0]$. 
    If we fix any $0 < r_1< r_2 \leq r_0$, by our assumption, $\abs{\frac{t_i}{r^2t}} + |x_i| \leq  \frac{\varepsilon}{100}$ for all sufficiently large $i$ whenever $t\in \left[-2,-\frac{1}{2}\right]$ and $r \in [r_1, r_2]$. 
    Therefore, \eqref{Shifting-the-flow-1} implies that $\frac{1}{\sqrt{-t}} M_t^{\pr{x_i,t_i,r^{-1}}}$ is $\varepsilon$-close to $\Gamma_{-1}$ for all sufficiently large $i$ whenever $t\in \left[-2,-\frac{1}{2}\right]$ and $r \in (r_1, r_2]$. 
    \end{proof}
    
    From now on, we assume $ 0 < \varepsilon < \frac{1}{100n}.$
    Take $r_0$ from Claim~\ref{claim}. 
    We define
    \begin{align*}
        \rho_i = \inf\left\{\rho \leq r_0 : \ \begin{aligned}
        &\text{for each } r\in (\rho,r_0), \ \frac{1}{\sqrt{-t}} M_t^{\pr{x_i,t_i,r^{-1}}} \text{ is } \varepsilon\text{-close to}\\ & \text{a cylinder in } \C_k^0 \text{ when }t \in [-2,-1/2] 
    \end{aligned}\right\}
    \end{align*}
    and 
    \begin{align*}
        \underline{\rho}_i = \sup \set{r < r_0  : \ M_{-1}^{\pr{x_i,t_i,r^{-1}}} \text{ is } \varepsilon \text{-close to a cylinder in } \C_\ell^0 }.
    \end{align*}
    Recall that $\C^0_k$ is the space of all $k$-cylinders centered at the origin.
    Again, in both definitions, we don't specify the directions of the cylinders.
    For each $m  > 1$, by Claim~\ref{claim}, we have $\rho_i \leq \frac{r_0}{m}<r_0$ for large $i$. 
    Therefore, $\limsup\limits_{i\to\infty} \rho_i \leq \frac{r_0}{m}$. 
    Letting $m\rightarrow\infty,$ we have
    \begin{align}\label{Cylindrical-scale-tends-to-0}
        \limsup_{i\to\infty}\rho_i  = 0.
    \end{align}

    Since $M_t^{\infty}$ is a shrinking $\ell$-cylinder, its axis keeps a fixed distance to the origin. 
    If we scale it by $\frac{1}{\sqrt{-t}}$ then the axis will get closer to the origin when $-t$ gets larger. 
    As a result, there exists $\mathcal{T} = \mathcal T(\varepsilon) \geq 1$ such that $\frac{1}{\sqrt{-t}}M_t^{\infty}$ is $\frac{\varepsilon}{2}$-close to an $\ell$-cylinder for all $t \leq -\mathcal{T}$.
    We emphasize again that there is no specification for the directions of axes $\mathbb{R}^{\ell}$ here.
    By the smooth convergence, the rescaled flow $\frac{1}{\sqrt{\mathcal{T}}}M_{-\mathcal{T}}^{\pr{x_i,t_i,\lambda_i}}$ is $\varepsilon$-close to an $\ell$-cylinder for all large $i$. 
    Since 
    $\frac{1}{\sqrt{\mathcal{T}}} M_{-\mathcal{T}}^{\pr{x_i,t_i,\lambda_i}} 
    = M_{-1}^{\pr{x_i,t_i,\lambda_i\sqrt{\mathcal{T}}^{-1}}},$ we have $\underline{\rho}_i \geq \lambda_i^{-1}\sqrt{\mathcal{T}} > 0$ for all large $i$.
    On the other hand, the definitions of $\rho_i$ and $\underline{\rho}_i$ imply $\underline{\rho}_i < \rho_i.$
    Combining these, we have
    \begin{align*}
        0 < \underline{\rho}_i < \rho_i
    \end{align*}
     for all large $i$.
    There are then two possible cases. 
    We will derive a contradiction in each case.
    
    \noindent\textbf{Case 1:} $\limsup\limits_{i\to\infty} \frac{\rho_i}{\underline{\rho}_i} < \infty$. 
    
    We may pass to a subsequence, still denoted by $i$, such that 
    $\lim\limits_{i\to\infty} \frac{\rho_i}{\underline{\rho}_i} = \Lambda \geq 1$.
    Consider the sequence of rescaled flows $M_t^{\pr{x_i,t_i, \rho_i^{-1}}}$ that converges to $M_t^{a}$. 
    By the non-degeneracy, \eqref{Cylindrical-scale-tends-to-0}, and the definition of $\rho_i$, the limit $M_t^{a}$ must be a shrinking cylinder that is $2\varepsilon$-close to a $k$-cylinder centered at the origin at time $-1$. 
    Note that this cylinder might be a rotation of $\Gamma_{-1}$. 
    Therefore, the limit flow $M_t^a$ must be a shrinking $k$-cylinder.

    Now, we consider the rescaled flow $M_t^{(x_i,t_i,\underline{\rho}_i^{-1})}$.
    The fact that $\lim\limits_{i\to\infty} \frac{\rho_i}{\underline{\rho}_i} = \Lambda$ and the definition of $M_t^a$ implies that $M_t^{(x_i,t_i,\underline{\rho}_i^{-1})}$ converges to $\Lambda M_{\Lambda^{-2}t}^a$. 
    On the other hand, the definition of $\udl{\rho}_i$ implies that the limit (which is $\Lambda M_{\Lambda^{-2}t}^a$) must be $2\varepsilon$-close to an $\ell$-cylinder in $\mathcal{C}^0_{\ell}$ at $t = -1$. 
    In other words, some time slice of a shrinking $k$-cylinder is $2\varepsilon$-close to an $\ell$-cylinder in $C_{\ell}^0$. 
    This cannot happen when $\varepsilon$ is small enough.
    (For example, we can take $\varepsilon<(100n)^{-1}$ to prevent this.)

    \noindent\textbf{Case 2:} $\limsup\limits_{i\to\infty} \frac{\rho_i}{\underline{\rho}_i} = \infty$. 
    
    We may pass to a subsequence, still denoted by $i$, such that 
    $\lim\limits_{i\to\infty} \frac{\rho_i}{\udl{\rho}_i} = \infty$.
    By the definition of $\udl{\rho}_i,$ we have
    \begin{align*}
        \Theta_{M_t} \pr{x_i,t_i,\underline{\rho}_i}  
        \geq \lambda\pr{S^\ell}- C\varepsilon 
    \end{align*}
    for all large $i$, where $C$ depends only on the dimension and the entropy of the initial data.
    In particular, for all sufficiently small $\varepsilon,$ we have 
    \begin{align}\label{Dnesity-lower-bound}
        \Theta_{M_t} \pr{x_i,t_i,\underline{\rho}_i}  > \frac{4}{3}
    \end{align}
    for all large $i$ by Lemma \ref{lem:ent}.
    Consider the sequence of rescaled flows $M_t^{\pr{x_i,t_i, \rho_i^{-1}}}$ that converges to $M_t^{b}$. 
    Again, $M_t^{b}$ must be a shrinking cylinder that is $2\varepsilon$-close to a $k$-cylinder centered at the origin at time $-1$. 
    This cylinder might be a rotation of $\Gamma_{-1}$. 
    By \eqref{Dnesity-lower-bound} and the fact that $\lim\limits_{i\to\infty} \frac{\rho_i}{\underline{\rho}_i} = \infty,$ we have a density lower bound at $(0,0)$ given by
    \begin{align*}
        \Theta_{M_t^b}(0,0) \geq \frac{4}{3}.
    \end{align*}
    In particular, $(0,0)$ is not a regular point of $M_t^b$. 
    This implies that $M_t^b$ must be a shrinking $k$-cylinder centered at the origin that extincts at time $t=0$. 
    Hence, $\frac{1}{\sqrt{-t}}M_{t}^b$
    is $0$-close to a $k$-cylinder for all $t < 0$. 
    However, the definition of $\rho_i$ implies that $\frac{1}{\sqrt{-t}}M_{t}^{b}$ is not $\frac{\varepsilon}{2}$-close to any $k$-cylinder for some $t \in \left[-2,-\frac{1}{2}\right]$. 
    This is a contradiction.

    So far, we have proved that any limit flow $M_t^{\infty}$ is either $\mathbb{R}^n$ or a $k$-cylinder with multiplicity one. 
    It remains to prove that when we get a $k$-cylindrical limit flow, the axis aligns with the tangent flow $\Gamma_t$.

    \noindent\textbf{Part 2:} Uniqueness of axes.
    
    We first show that at scale $\rho_i$ and above, not only does the flow look close to a $k$-cylinder centered at origin, but also the axis of the cylinder almost aligns with the tangent flow. 

    To this end, consider the rescaled mean curvature flow 
    \begin{align*}
        \ovl{M}^i_{\tau} 
        = e^{\frac{\tau}{2}}M_{-e^{-\tau}}^{\pr{x_i,t_i, r_0^{-1}}}
    \end{align*}
    for $\tau>0.$
    By the definition of $\rho_i$ and the claim, $\ovl{M}_{\tau}^i$ is $\varepsilon$-close to a $k$-cylinder for all $\tau \in \left[0, 2\log\pr{\frac{r_0}{\rho_i}}+\log 2\right]$.
    By the entropy bound and Colding--Minicozzi's \L ojasiewicz inequality \cite{CM15}, the axis of the cylinder has total rotation at most $C\varepsilon$ during 
    $\tau \in \left[1, 2\log\pr{\frac{r_0}{\rho_i}}\right]$ 
    (cf. \cite[Proposition~4.1]{GH} and \cite[Lemma~A.48]{CM15}). 
    By Claim~\ref{claim}, $\ovl{M}_{\tau}^i$ is $\varepsilon$-close to $\Gamma_{-1}$ for $\tau\in [0,10]$ as long as $i$ is large enough. 
    Therefore, $\ovl{M}_{\tau}^i$ is $C\varepsilon$-close to $\Gamma_{-1}$ for all $\tau \in \left[0, 2\log\pr{\frac{r_0}{\rho_i}}\right]$. 
    In particular, for all large $i,$ we have
    \begin{align}\label{Cepsilon-close-to-Gamma}
        M_{-1}^{\pr{x_i,t_i,\rho^{-1}}} \text{ is } C\varepsilon\text{-close to } \Gamma_{-1}  
        \text{ when } \rho \in [\rho_i, r_0].
    \end{align}

    Next, we want to show that the curvature scale is comparable to $\rho_i$.
    That is,
    \begin{align} \label{Compare-regularity-scale-with-cylindrical-scale}
        \liminf \limits_{i\rightarrow \infty} \frac{1}{\rho_i H(x_i, t_i)} > 0.
    \end{align}
    Suppose that this is not the case. 
    Then, we consider $M_t^{\pr{x_i,t_i,\rho_i^{-1}}}$ 
    and a subsequential limit $M_t^{c}.$
    The limit $M_t^{c}$ must be a shrinking $k$-cylinder that is smooth when $t<0$ and has 
    $H^{-1}(0,0) = 0.$\footnote{We use the convention that $H^{-1}(x_0,t_0)=0$ when $(x_0, t_0)$ is a singularity of the flow.}
    In particular, 
    $\frac{1}{\sqrt{-t}} M_t^{c}$ 
    is $0$-close to a $k$-cylinder for all $t < 0$.  
    However, by the definition of $\rho_i$, the limit $\frac{1}{\sqrt{-t}}M_t^{c}$ is not $\frac{\varepsilon}{2}$-close to any $k$-cylinder for some $t \in \left[-2,-\frac{1}{2}\right].$
    This leads to a contradiction.

    Finally, we can finish the proof of the uniqueness of axis part.
    If $M_t^{\infty}$ is a multiplicity one plane, then we are done. 
    Thus, we assume that $M_t^{\infty}$ is a $k$-cylinder. 
    In this case, the axis $\mathbb{R}^{k}$ aligns with its blow-down limit. 
    Therefore, it suffices to prove that the blow-down limit of $M_t^{\infty}$ is $\Gamma_t$. 
    Note that $M_t^{\infty}$ may not coincide with $\Gamma_t$, but the axis $\mathbb R^k$ must align.
    
    We claim that 
    \begin{align}\label{Compare-mu-with-cylindrical-scale}
        L:= \limsup\limits_{i\to\infty} \lambda_i \rho_i  <\infty
    \end{align}
    where we recall that $\lambda_i$ is the scaling sequence leading to the limit flow $M_t^\infty.$
    If this were not the case, then there would be a subsequence $i_j$ such that $\lim\limits_{j\rightarrow\infty}\lambda_{i_j}\rho_{i_j} = \infty$. 
    By \eqref{Compare-regularity-scale-with-cylindrical-scale}, we have 
    $\lim\limits_{j\rightarrow\infty} \frac{\lambda_{i_j}}{H\pr{{x_{i_j},t_{i_j}}}}  
    = \infty.$
    In particular, the limit of $M_t^{i_j} = M_t^{\pr{x_{i_j},t_{i_j}, \lambda_{i_j}}}$ has $H(0,0)= 0$. On the other hand, the limit coincides with  $M_t^{\infty}$ and should be an $\ell$-cylinder, so the mean curvature should not vanish anywhere. We thus get a contradiction and proved \eqref{Compare-mu-with-cylindrical-scale}.

    From \eqref{Cepsilon-close-to-Gamma}, \eqref{Compare-regularity-scale-with-cylindrical-scale}, and \eqref{Compare-mu-with-cylindrical-scale}, we obtain that 
    $\frac{1}{L} M_{-L^2}^i = M_{-1}^{\pr{x_i,t_i,L^{-1} \lambda_i}}$ is $C\varepsilon$-close to $\Gamma_{-1}$ in a subsequence. 
    Therefore, a subsequence of $\frac{1}{L}M_{L^2t}^i$ converges to a $k$-cylinder that is $C\varepsilon$ close to $\Gamma_{t}$ at time $t = -1$. 
    This implies that $\frac{1}{L}M_{L^2t}^{\infty}$ is a $k$-cylinder which is $C\varepsilon$-close to $\Gamma_{-1}$ at time $-1$. 
    In view of the the evolution of cylinders, the blow-down limit of $M_t^{\infty}$ is $C\varepsilon$-close to $\Gamma_t$ at time $-1$.
    Since $\varepsilon$ is arbitrary, we conclude that the blow-down limit of $M_t^{\infty}$ coincides with $\Gamma_t$. 
    This finishes the proof.
\end{proof}

\begin{remark}\label{rmk:non-deg-arnold-thom}
    Theorem~\ref{lem:nondeg-unique} does not assume that the flow is mean convex.
    Thus, one can see the arguments in Section~\ref{sec:4.2} and Section~\ref{sec:4.3} also work when the flow is not globally mean convex, since the non-degenerate condition guarantees a mean convex neighborhood near such a singularity.
    The only issue is that one has to use the definition of flow lines given in the context of general flows, cf. Section~\ref{sec:general-flow}.
\end{remark}

Before we end this section, we note that in \cite{SX4}, Sun--Xue defined a different notion of non-degenerate cylindrical singularities and proved that they are generic. 
Their non-degeneracy condition
implies the one defined in Definition~\ref{def:nondeg}.
Also, there are higher dimensional dumbbells having non-degenerate cylindrical singularities.
In general, examples of higher dimensional cylindrical singularities were studied in \cite{W13}.


\subsection{Uniqueness of limit points}
\label{sec:4.2}
With the uniqueness result for limit flows of non-degenerate singularities (Theorem~\ref{lem:nondeg-unique}), we will first see how the \L ojasiewicz theorem, that is, the uniqueness of limit points of a flow line, follows from it.

\begin{thm}[\L ojasiewicz for non-degenerate cylindrical singularities]
\label{thm-L2}
    If $\gamma$ has a limit point which is a non-degenerate cylindrical singularity,
    then this limit point is the unique limit point of~$\gamma.$
\end{thm}

The key idea is that Theorem~\ref{lem:nondeg-unique} guarantees the existence of a ``canonical neighborhood,'' inside which every regular point behaves nicely, that is, like one on a cylindrical flow.
This prevents the occurrence of multiple limit points.
As mentioned in Remark~\ref{rmk:non-deg-arnold-thom}, since Theorem~\ref{lem:nondeg-unique} does not require mean convexity, we can also extend Theorem~\ref{thm-L2} to a general flow with a non-degenerate cylindrical singularity if we use a general notion of flow lines, cf. Definition~\ref{def:new-flow-line}.

\begin{proof}
	We assume the origin is a limit point of $\gamma$ and is a singularity of the flow modeled on a $k$-cylinder.
    With a reparametrization (using the time of the flow) and a translation, we assume $\gamma(t_i)\to 0$ for some $t_i\to 0.$ 
	Our goal is to show $\gamma(t)\to 0$ as $t\to 0.$
	We will prove this by showing that $\gamma$ behaves like a cylindrical flow lines for $t\in[s_i, t_i]$ for some $s_i<t_i,$
	and use this cylindricality to show that we can in fact take $s_i=t_{i-1}.$	
	Without loss of generality, we may assume that $\{t_i\}$ is a strictly increasing sequence.

	Fix any $\delta > 0$, for $i\geq 2$, we let
	\begin{align}\label{si-def}
		s_i = s_i(\delta) = \inf \set{\tau \in[t_{i-1}, t_i]
			: |\gamma(t)| \le |\gamma(t_i)| 
			+\delta
			\text{ for all } t \in [\tau, t_i]}.
	\end{align}
	We claim that the flow behaves like a shrinking cylinder along $\gamma$ when  $t\in[s_i,t_i].$
	\begin{claim}\label{Loj;claim}
		Given any $\varepsilon>0,$ there exists $\delta_0 = \delta_0(\varepsilon) > 0$ with the following property. 
        If $\delta \leq \delta_0$, then for all large $i$ and $t\in[s_i(\delta), t_i],$ we have $H\pr{\gamma(t),t} > 0$
		and
		\begin{align}\label{lojasiewicz;H-time-derivative}
			\sup\limits_{t\in [s_i(\delta),t_i]} \abs{\frac{\partial_t H}{H^3}\pr{\gamma(t),t} - \frac{1}{n-k}} \leq \varepsilon.
		\end{align}
	\end{claim}
	
	\begin{proof}[Proof of Claim~\ref{Loj;claim}]
		If the claim were false, then there would exist $\varepsilon>0,$ $\delta_i\rightarrow 0$, $j_i\geq i$  and $\tau_i\in [s_{j_i}(\delta_i),t_{j_i}]$ such that 
        \begin{align}\label{Loj;contradiction;H-diff}
	    \abs{\frac{\partial_t H}{H^3}\pr{\gamma(\tau_i),\tau_i)} - \frac{1}{n-k}} > \varepsilon
    \end{align}
    or
    \begin{align}\label{Loj;contradiction;H=0}
	    H(\gamma(\tau_i),\tau_i) = 0
    \end{align}
		 for each $i$. 
		We may pass to a subsequence (still denoted by $i$) so that $j_i$ is strictly increasing in $i$.
        To get a contradiction, we consider
    \begin{align*}
	    M_t^{i} 
        = M_t^{\pr{\gamma(\tau_i), \tau_i, H(\gamma(\tau_i), \tau_i)}}. 
    \end{align*} By definition we have $\tau_i\rightarrow 0$. Moreover, 
	\begin{align*}
		|\gamma(\tau_i)| \leq |\gamma(t_{j_i})| + \delta_i \rightarrow 0
	\end{align*}
	Thus $\gamma(\tau_i)\rightarrow 0$, which is a limit point, so $H(\gamma(\tau_i), \tau_i)\rightarrow \infty$ because $\n u(0)=0.$  
		By Brakke's compactness theorem \cite{B78} (cf. \cite{HK1}), after passing to a subsequence (still denoted by $i$), $M_t^{i}$ converges to an integral Brakke flow $M_t^{\infty}.$ 
		By the non-degeneracy assumption and Theorem~\ref{lem:nondeg-unique}, $M_t^{\infty}$ is a shrinking $k$-cylinder with $H(0,0) = 1$. 
		By the evolution of cylinders, we know that
		$\partial_t H = \frac{H^3}{n-k}$ on $M_t^{\infty}$ for $t<\sqrt{\frac{n-k}{2}}$. 
		Brakke's regularity theorem \cite{B78} (cf. \cite{W05, HK1}) guarantees that the convergence is smooth on compact sets in $\mathbb{R}^{n+1}\times (-\infty,0]$. 
    Thus, for all large $i,$ we have
 	\begin{align*}
 	      \abs{\frac{\partial_t H}{H^3}\pr{\gamma(\tau_i), \tau_i} -\frac{1}{n-k}}
          \le \varepsilon
            \text{ and } H\pr{\gamma(\tau_i), \tau_i} > 0.
 	\end{align*}
    This means that neither \eqref{Loj;contradiction;H-diff} nor \eqref{Loj;contradiction;H=0} holds for large $i$, and we get a contradiction.
	\end{proof}

	We continue the proof of Theorem~\ref{thm-L2}. 
	We specifically choose $\varepsilon=1/n^2$ and take $\delta_0$ to be given by Claim~\ref{Loj;claim}. 
    We then work with any fixed $\delta < \delta_0$ and write $s_i=s_i(\delta).$
    Then we have
	\begin{align*}
		\partial_t H\pr{\gamma(t),t}
		\geq \frac{1}{n}H^3\pr{\gamma(t),t}
		\text{ and }\quad H\pr{\gamma(t),t}> 0 
	\end{align*}
	when $t \in [s_i,t_i]$ for all large $i.$
	This differential inequality and the fact that $H(\gamma(t_i),t_i) > 0$ imply 
	\begin{align*}
		\frac{d}{dt} H(\gamma(t),t)^{-2} \leq -\frac{2}{n}
	\end{align*}
	when $t \in [s_i,t_i].$
	Integrating from $t$ to $t_i,$ we obtain that
	\begin{align*}
		H(\gamma(t),t)^{-2} 
		\geq \frac{2(t_i-t)}{n} 
	\end{align*}
	for all $t\in [s_i,t_i]$. 
	Using this inequality and the differential equation for $\gamma,$ that is, $|\gamma'(t)| = H(\gamma(t),t)$ (cf. Remark~\ref{rmk:gamma-para}), we can integrate along $t\in [s_i,t_i]$ and obtain
	\begin{align}\label{Loj;dist;estimate}
		|\gamma(s_i)| 
		\leq |\gamma(t_i)| + \int_{s_i}^{t_i}\sqrt{\frac{n}{2(t_i-t)}} dt
		=  |\gamma(t_i)| + \sqrt{2n(t_i-s_i)}
        < |\gamma(t_i)| + \delta
	\end{align}
	where we use that for all large $i$, we have 
	\begin{align*}
		\sqrt{2n(t_i-s_i)} \leq \sqrt{2n(t_i-t_{i-1})} <\delta.
	\end{align*}
    On the other hand, by the definition of $s_i,$ if $s_i\neq t_{i-1},$ then $|\gamma(s_i)| = |\gamma(t_i)| + \delta.$
    This contradicts \eqref{Loj;dist;estimate}.
    Thus, we deduce that $s_i =t_{i-1}$ for all large $i$. 
	This then implies
	\begin{align*}
		\gamma([t_{i-1},t_i])\subset B_{|\gamma(t_i)| + \delta}(0) 
	\end{align*}
	for all large $i$. 
	Because $|\gamma(t_i)| \to 0,$ we have
	\begin{align*}
		\gamma\pr{[t_{i-1},t_i]}\subset B_{2\delta}(0)
	\end{align*}
	for all large $i$. 
	Since $t_i\nearrow 0,$ this means that
	\begin{align*}
		\gamma([t_{i},0))\subset B_{2\delta}(0)
	\end{align*}
	for all large $i$. 
	Because $\delta < \delta_0$ can be arbitrarily small, we conclude that
	$\lim\limits_{t\rightarrow 0^-}\gamma(t) = 0.$
	This finished the proof of the uniqueness of limit points of $\gamma.$
\end{proof}

\begin{remark}
    \label{rmk:another-proof-of-L}
    The proof of Theorem~\ref{thm-L2} may lead to another proof for Theorem~\ref{thm-L} when the flow line is cylinder-type.
    Without knowing Theorem~\ref{thm-L} a priori, one may still prove Theorems~\ref{cylindrical-limit-flow} and~\ref{bowl-limit-flow}.
    The possibility of having more than one limit point may complicate the proof, requiring the \L ojasiewicz estimate by Colding--Minicozzi in the argument again.
    After the uniqueness of limit flows is proven in this way, to prove the uniqueness of limit points, the case of cylinder-type flow lines can be treated as in the proof of Theorem~\ref{thm-L2}, while the case of bowl-type flow lines can then be dealt with directly, as our treatment of Case 1 in the proof of Theorem~\ref{thm:AT-neck}.
\end{remark}

\begin{remark}\label{rmk:sec34-cyl}
    The techniques used to prove Claim~\ref{Loj;claim} share similarities with those in Lemma~\ref{lem:cylinder-type-case} (cf. equation~\eqref{main;gamma-distance-up-bound}), as both rely on the differential inequalities for $H$ derived from the Jacobi equation and approximation by a cylinder.
    A key difference, however, is that in Lemma~\ref{lem:cylinder-type-case}, the uniqueness of limit points is already known, and equation~\eqref{main;gamma-distance-up-bound} is used to establish the angle estimate~\eqref{gamma(t)-in-Cd}. 
    In contrast, in the proof of Theorem~\ref{thm-L2}, the cylindrical estimate must be handled more delicately due to the absence of such uniqueness.
\end{remark}


\subsection{Uniqueness of limit tangents}
\label{sec:4.3}
We will finish the proof of the Arnold--Thom conjecture for non-degenerate singularities.

\begin{thm}[Arnold--Thom for non-degenerate singularities]\label{thm:AT-cyl}
    Let $M_t$ be a closed mean convex mean curvature flow in $\bb R^{n+1}$ and let $u$ be its arrival time function.
    If $\gamma$ is a gradient flow line for $u$ with a limit point being a non-degenerate cylindrical singularity,
    then the limit point of $\gamma$ is unique and the unit tangents of $\gamma$ converge.
\end{thm}

\begin{proof}
    The uniqueness of limit points is addressed in Theorem \ref{thm-L2}.
    For the uniqueness of limit tangents,
    by Theorem \ref{lem:nondeg-unique}, at a non-degenerate cylindrical singularity, all the limit flows are the same type of cylinders with parallel axes.
    Thus, applying Lemma~\ref{lem:cylinder-type-case} finishes the proof.
\end{proof}


\section{\bf General flows with neck singularities}
\label{sec:general-flow}

In this section, we discuss neck singularities even if a mean curvature flow is not mean convex.
In this case, the flow may fatten, so it is subtle to determine whether the flow has a neck singularity.
We follow the treatments in \cite{CHHW} (cf. \cite{I92, HW}).

\subsection{Outer flows}
Let $K$ be an arbitrary closed set in $\bb R^{n+1}.$
Then, by the works of \cite{ES, CGG, I94}, there exists a unique level set flow $F_t(K)$ starting from $K.$
It is the maximal family of closed sets satisfying the avoidance principle (when compared with closed smooth mean curvature flows) and starting with $K.$

Let $M$ be a closed smooth hypersurface in $\bb R^{n+1},$ and let $K$ be the compact domain with $\bd K = M.$
The level set flows we will consider are
\begin{align*}
    K(t) &:= F_t(K)
    \text{ and }\\
    K'(t) &:= F_t\pr{\ovl{\bb R^{n+1}\setminus K}}
\end{align*}
starting from $K$ and $\ovl{\bb R^{n+1}\setminus K}.$
Let $\mathcal K$ and $ \mathcal K'$ be the corresponding space-time track, and we define
\begin{align*}
    M_t &:= \set{x\in\bb R^{n+1}: (x,t)\in\bd\mathcal K}
    \text{ and }\\
    M_t' &:= \set{x\in\bb R^{n+1}: (x,t)\in\bd\mathcal K'},
\end{align*}
called the outer and inner flows starting from $M.$
By \cite[Proposition~A.3]{HW}, both $M_t$ and $M_t'$ are weak set flows.
Then we can follow \cite[Definition~1.16]{CHHW} to define neck singularities for such flows.

\begin{definition}[Neck singularity]
    Let $M$ be a closed embedded hypersurface in $\bb R^{n+1}.$
    We say $(x,t)$ is an inward neck singularity of the evolution of mean curvature flow starting from $M$ if the rescaled flows $M_t^{x,t,\lambda_i}$ converge locally smoothly to a shrinking cylinder, where $M_t$ is the outer flow starting from $M.$
\end{definition}

One can similarly define what an outward neck singularity is using the inner flow of $M.$
The theories for inner and outer flows are parallel, and in the rest of this section, we only state and prove the results for outer flows and inward neck singularities.

Note that when the usual mean curvature flow starting from a closed hypersurface $M$ is smooth, the ordinary mean curvature flow then coincides with both the inner and outer flows before the first singular time. 
(That's why we use the same notation $M_t$ to denote an outer flow.)
When they are different, we will look at one of them to discuss neck singularities as above.
At the same time, we also have to say what a flow line is.
To this end, we recall \cite[Theorem~B.3]{HW} which says that there exists a Brakke flow $\mu_t$ starting from $M$ whose support is given by $M_t.$
We will write the generalized mean curvature vector of $\mu_t$ by $\h_{M_t},$ using which we can define the notion of flow lines in this weak setting.

\begin{definition}
    [Flow lines]
    \label{def:new-flow-line}
    Let $M$ be a closed embedded hypersurface in $\bb R^{n+1}$ and let $M_t$ be the outer flow starting from $M.$
    We say that a Lipschitz curve $\gamma\colon[t_1, t_2]\to\bb R^{n+1}$ is a flow line of $M_t$ if it satisfies
    \begin{align*}
        \gamma'(t) = \h_{M_t}\pr{\gamma(t)}
    \end{align*}
    for $t\in [t_1, t_2]$ such that $\h_{M_t}\pr{\gamma(t)}$ exists.
\end{definition}

Note that, a priori, we know that $\h_{M_t}$ is defined almost everywhere.
This allows us to have the definition above, but we do not have any existence and uniqueness theorems as in the usual ordinary differential equation theory.
In fact, a flow line can have strange behavior, especially at those points where the mean curvature of the flow is not defined.
However, the mean convex neighborhood theorem \cite[Theorem~9.1]{CHHW} implies that near an inward neck singularity, the flow is mean convex. 
Therefore, close to the neck singularity, $\h_{M_t}\pr{\gamma(t)}\neq 0$ for $t<\infty,$ so we can revert to the classical setting when approaching a neck singularity.
The rest of the proofs in Section~\ref{sec:AT-neck} are valid, and we get the \L ojasiewicz theorem and Arnold--Thom conjecture in the case of outer flows.

\begin{thm}\label{thm:AT-general}
    Let $M_t$ be the outer flow starting from a closed smooth hypersurface $M$ in $\bb R^{n+1}$ and let $(0,0)$ be an inward neck singularity of $M_t.$
    If $\gamma$ is a gradient flow line of $M_t$ with a limit point,
    then the limit point is unique and the unit tangents of $\gamma$ converge.
\end{thm}

In fact, with Definition~\ref{def:new-flow-line}, the \L ojasiewicz theorem and the Arnold--Thom conjecture also hold for a non-degenerate cylindrical singularity in a general mean curvature flow even without the mean convexity condition.
See Remark~\ref{rmk:non-deg-arnold-thom}.

We remark that gradient flow lines in the context of Theorem~\ref{thm:AT-general} may not be globally $C^1.$
However, the theorem only cares about the behaviors of flow lines near inward neck singularities.
In a small neighborhood of an inward neck singularity, the flow is a mean convex level set flow so a gradient flow line can be interpreted in the classical sense.
This again manifests the principle proposed by Colding--Minicozzi \cite{CM19} that solutions with low regularity can behave like analytic functions.

\subsection{Flows of two-spheres and generic hypersurfaces}
We mention three situations in which the general framework developed in the previous section applies directly without the issue of fattening.
All flows are assumed to be embedded.
The first case is when a flow starts from a two-sphere, and the second and the third cases are when a flow starts from a generic closed hypersurface.

With the works of \cite{W95, Br16, HW, CHH, BK23}, level set flows starting from a two-sphere are understood well.
In fact, this was already discussed in \cite[Proof of Theorem~1.12]{CHH} assuming the multiplicity one conjecture, which was recently proven by Bamler--Kleiner \cite{BK23} (cf. \cite[Corollary 1.3]{BK23}).
In particular, the level set flow starting from a two-sphere does not fatten.
Thus, starting from a smooth two-sphere, its inner flow, its outer flow, and its level set flow will all be the same.
Theorem~\ref{thm:AT-general} then implies the following uniqueness results (\L ojasiewicz theorem and Arnold--Thom conjecture).

\begin{cor}
    Let $M_t$ be a mean curvature flow starting from a two-sphere in $\bb R^{3}.$
    If $\gamma$ is a flow line of $M_t$ with a limit point,
    then the limit point is unique and the unit tangents of $\gamma$ converge.
\end{cor}

The second situation is related to the study of generic mean curvature flows of surfaces.
The program was initiated by Colding--Minicozzi \cite{CM12}, where they proved that generalized cylinders are the only linearly stable singularity models.
The possibilities of getting generic singularities by perturbing initial data were confirmed in \cite{CCMS, CCS, SX1, SX2}, with some potential scenario of high multiplicities.
Again, with the recent work of Bamler--Kleiner \cite{BK23}, combined with the work of Chodosh–Choi--Mantoulidis–Schulze \cite{CCMS} and  Chodosh–Choi–Schulze \cite{CCS}, it is now clear that mean curvature flows starting from generic closed embedded surfaces only develop neck or spherical singularities.
Thus, Theorem~\ref{thm:AT-general} again implies the following uniqueness results.

\begin{cor}
    Let $M_t$ be a mean curvature flow starting from a generic closed embedded surface in $\bb R^{3}.$
    If $\gamma$ is a flow line of $M_t$ with a limit point,
    then the limit point is unique and the unit tangents of $\gamma$ converge.
\end{cor}

We can also deal with generic low-entropy mean curvature flows in $\bb R^{n+1}$ with $n+1\in\set{4,5,6}$. 
By the work of Chodosh–Choi--Mantoulidis–Schulze \cite{CCMS2} and Chodosh--Mantoulidis–Schulze \cite{CMS}, if the entropy of the initial hypersurface is low, then the generic mean curvature flow only encounters cylindrical singularities.
In particular, Theorem \ref{thm:AT-general} implies the following result when the entropy is not greater than that of a bubble sheet $S^{n-2}\times \bb R^2.$

\begin{cor}
     Let $M_t$ be a mean curvature flow starting from a generic closed embedded hypersurface $M^n\subset \bb R^{n+1}$ with $n+1\in\set{4,5,6}$ and $\lambda(M) \leq \lambda(S^{n-2}\times \bb R^2)$.
    If $\gamma$ is a flow line of $M_t$ with a limit point,
    then the limit point is unique and the unit tangents of $\gamma$ converge.
\end{cor}

\appendix

\section{Convergence of flow line}
\label{app:A}

In this appendix, we mention a general convergence result for flow lines in a sequence of mean curvature flows with a limit.

\begin{lemma}\label{lem:flow-line-conv-general}
    Let $M^i_t\sbst\bb R^{n+1}$ be a sequence of mean curvature flows converging in $C^\infty_{\rm loc}$ to a smooth mean curvature flow $M^\infty_t$ $\pr{t\in \left(-\ovl T,0\right]}$.
    Suppose $\gamma^i(t)$ $\pr{t\in \left(-\ovl T,0\right]}$ is a flow line in $M^i_t$ with $\gamma^i(0)$ converging to $x^\infty\in M^\infty_0,$
    and suppose that 
    there exists a flow line $\gamma^\infty(t)$ $\pr{t\in \left(-\ovl T,0\right]}$ of $M^\infty_t$ with $\gamma^\infty(0)=x^\infty.$ 
    Then $\gamma^i(t)$ converges to $\gamma^\infty(t)$ on any compact subsets of $(-\ovl T,0].$
\end{lemma}

The assumption that $\gamma^{\infty}(t)$ exists on $(-\ovl{T},0]$ is automatically satisfied in many situations. 
For example, it is true when $M_t^{\infty}$ has bounded mean curvature or when the limit flow $M_t^{\infty}$ arises as a limit flow of a mean convex mean curvature flow.
In our case, we can apply Lemma~\ref{lem:flow-line-existence} to get the existence of the limit flow lines.

\begin{proof}
    We may assume $x^\infty=0$ by translations.
    To show the convergence of $\gamma^i$ to $\gamma^\infty,$ we fix $T<\ovl T$ and take an $\varepsilon$-neighborhood $U_\varepsilon$ of the flow line $\gamma^\infty$ for a fixed $\varepsilon > 0.$
    That is, we let
    \begin{align*}
        U_\varepsilon:=
        \bigcup_{t\in[-T,0]} 
        B_{\varepsilon} \pr{\gamma^\infty(t)}.
    \end{align*}
    For each $i,$ we let 
    \begin{align*}
        s_i
        := \inf\set{
        \tau\in[-T, 0]:
        \gamma^i\pr{[\tau, 0]}\sbst U_\varepsilon
        }.
    \end{align*}
    After passing to a subsequence, we may assume $-T_\varepsilon
    := \lim\limits_{i\to\infty}
    s_i$
    exists.

    If it were not true that $T_\varepsilon=T,$
    then we would have $T_\varepsilon<T.$
    Thus, $\gamma^i\pr{[-T_\varepsilon,0]}$ is contained in $U_\varepsilon$ for $i$ large.
    The locally smooth convergence implies the convergence of $\gamma^i$ to $\gamma^\infty$ on $[-T_\varepsilon, 0].$
    In particular, $\gamma^i(-T_{\varepsilon})\rightarrow \gamma^{\infty}(-T_{\varepsilon})$ and $s_i < -T_{\varepsilon}$ for all large $i$. 

    By the definition and the smoothness of $\gamma^i$, we have $\gamma^i(s_i) \in \partial U_{\varepsilon}$. 
    Thus, 
    \begin{align*}
        |\gamma^{i}(s_i)- \gamma^i(-T_{\varepsilon})| \geq \frac{\varepsilon}{2}
    \end{align*}
    for all large $i$. 
    Since $\gamma^{i}([s_i,0])\subset \ovl{U}_{\varepsilon}$, a compact set in $\mathbb{R}^{n+1}$, by the locally smooth convergence, we have a uniform curvature bound $|H|\leq C$ on $\gamma^i([s_i,0])$, so
    \begin{align*}
        |\gamma^{i}(s_i)- \gamma^i(-T_{\varepsilon})| \leq C|s_i - (-T_{\varepsilon})|.
    \end{align*}
    Thus, we obtain a lower bound on $|s_i - (-T_{\varepsilon})|,$ that is,
    \begin{align*}
        |s_i - (-T_{\varepsilon})| \geq \frac{\varepsilon}{2C}
    \end{align*}
    for all large $i$. 
    However, this then implies $\lim\limits_{i\rightarrow\infty} s_i < -T_{\varepsilon}$, which is a contradiction.

    Thus, we prove that $T_\varepsilon = T,$ and hence the convergence of $\gamma^i$ on $[-T,0]$ follows from the locally smooth convergence of the flows.
    This proves the convergence on any $[-T,0]\sbst (-\ovl T,0].$
\end{proof}

\section{Proof of Proposition~\ref{prop:convergence}}
\label{sec:CHHW+SS}

In this appendix, we give a streamlined proof of Proposition~\ref{prop:convergence}.

\begin{proof}[Proof of Proposition~\ref{prop:convergence}]
Without loss of generality, we may assume that $p= 0$ is the origin and the unique tangent flow at $(0,0)$ is the neck $\Gamma = \big\{(x_1,...,x_{n+1}) \ | \  x_1^2+\cdots+x_n^2 = 2(n-1)\big\}$ with axis  $\ell = \big\{(x_1,...,x_{n+1}) \ | \  x_1= x_2 =\cdots = x_n =0\big\}.$

By the definition of tangent flows, for each $\varepsilon,$ there is a constant $\rho = \rho(\varepsilon) > 0$ with $\rho(\varepsilon)\rightarrow 0$ as $\varepsilon \rightarrow 0$ such that for all $0 < r\leq \rho$ we have that
\begin{center}
	$\frac{1}{\sqrt{-t}} M_t^{(0,0,r^{-1})}$ is $\frac{\varepsilon}{4}$ close to $\Gamma$ when $t\in [-2,-\frac{1}{2}].$
\end{center}

We define  $Z^{\rho(\varepsilon)}(\bar{x},\bar{t})$ to be 
\begin{align*}
	Z^{\rho(\varepsilon)}(\bar{x},\bar{t}) := \inf
    \left\{z: 
    \begin{aligned}
    &\frac{1}{\sqrt{-t}} M_{t}^{(\bar{x},\bar{t},r^{-1})} \text{ is } \varepsilon \text{-close to some neck in }\mathcal{C}_1^0 \\ &\text{ for each } r\in (z,\rho(\varepsilon)) \text{ and } t\in \left[-2,-1/2\right]
    \end{aligned}
    \right\}.
\end{align*}
We also define $Z_i = Z^{\rho(\varepsilon)}\pr{\gamma(t_i),t_i}$.
Since $(\gamma(t_i),t_i)$ is regular point, it follows that $Z_i >0$.
Now we rescale the flow by $Z_i^{-1}$ around $(\gamma(t_i),t_i)$. 
Namely, we consider 
\begin{align*}
	M^i_t := M_t^{(\gamma(t_i),t_i,Z_i^{-1})}.
\end{align*}
Then, for each $0<r\leq \rho(\varepsilon)$, there exists $i_0 = i_0(r)$ such that for all $i\geq i_0$ we have that 
\begin{align*}
    \frac{1}{\sqrt{-t}}M_{t}^{\pr{\gamma(t_i),t_i,s^{-1}}}
    \text{ is }
    \frac\varepsilon 2
    \text{-close to }
    \Gamma
    \text{ for all }
    s\in [r,\rho(\varepsilon)]
    \text{ when }
    t\in \left[-2,-\frac{1}{2}\right].
\end{align*}
By definition, this implies that $\limsup\limits_{i\rightarrow\infty}Z_i \leq r$. 
Since $r$ is arbitrary,  we have $\lim\limits_{i\to\infty} Z_i= 0$ and $\lim\limits_{ i\rightarrow \infty}\frac{\rho(\varepsilon)}{Z_i} =\infty.$
By the convergence theorem of White (c.f \cite{W00, W03, HK1}), $M_t^i$ converges smoothly to a noncollapsed ancient solution $M_t^{\infty}$ such that $0\in M_0^{\infty}.$ 
By the classification result of \cite{BC19, BC21}, the definition of $Z$, and Lemma \ref{lem:no-bowl}, we know that $M_t^{\infty}$ must be a shrinking neck or a bowl soliton with axis $\ell'$.
Moreover,  $M^{\infty}_t$ is not $\frac{\varepsilon}{2}$-close to any neck in $\mathcal{C}_1^0$ for some $t\in \left[-2,-\frac{1}{2}\right]$.

We first claim that $\ell'$ is at most $C\varepsilon$ away from $\ell$. 
(Namely, the angle between $\ell'$ and $\ell$ is at most $C\varepsilon.$)
In fact, by the definition of $Z$, we have $\ovl{M}_{\tau}^i = e^{\frac{\tau}{2}}M_{-e^{-\tau}}^i$ is $\varepsilon$-close to $\Gamma$ for all $\tau \in \left[-\ln 2 - \ln \frac{\rho(\varepsilon)}{Z_i}, \ln 2\right]$.  
By Colding--Minicozzi's \L ojasiewicz inequality \cite{CM15}, the axis of the cylinder has total rotation at most $C\varepsilon$ during 
    $\tau \in \left[ - \ln\frac{\rho(\varepsilon)}{Z_i}, 0\right]$ 
    (c.f. \cite[Proposition~4.1]{GH} and \cite[Lemma~A.48]{CM15}). 
    By definition of $M^i$ and $\rho(\varepsilon)$, $\ovl{M}_{\tau}^i$ is $\varepsilon$-close to $\Gamma$ when $\tau = -\ln\frac{\rho(\varepsilon)}{Z_i}$ as long as $i$ is large enough. 
    Therefore, $\ovl{M}_{\tau}^i$ is $C\varepsilon$-close to $\Gamma$ for all $\tau \in \left[ - \ln\frac{\rho(\varepsilon)}{Z_i}, 0\right]$.  
    Passing to the limit proves the claim.

We next claim that $C^{-1} < H(\gamma(t_i),t_i)Z_i \leq C$. 
If this is not the case, then passing to the limit, we have $H(0,0) = 0$ or $H^{-1}(0,0) = 0$ in $M^{\infty}_t$. 
The first option is not consistent with the only two possibilities of $M^{\infty}_t$. 
For the second case, $M^{\infty}_t$ must be a shrinking neck with extinction time at 0 and with axis passing through the origin $0$. 
However, this contradicts the fact that $M^{\infty}_t$ is not $\frac{\varepsilon}{2}$-close to any neck in $\mathcal{C}_1^0$ for some $t\in \left[-2,-\frac{1}{2}\right]$. 

As a consequence, any subsequential limit   $\td{M}^{\infty}_t$ of $M_t^{(\gamma(t_i),t_i,H(\gamma(t_i),t_i)^{-1})}$  must be a parabolic rescaling of $M_t^{\infty}$, i.e.,  $\td{M}^{\infty}_t$ = $\lambda {M}^{\infty}_{\lambda^{-2}t}$ for some $\lambda \in [C^{-1},C]$. 
Hence, $\td{M}^{\infty}_t$ is still a shrinking neck or bowl soliton with axis at most $C\varepsilon$ away from $\ell$.

Since $\varepsilon$ is arbitrary, we conclude that (after passing to a further subsequence) $\td{M}^{\infty}_t$ is a shrinking neck or bowl soliton with its axis exactly $\ell$. 
This finishes the proof.
\end{proof}

\end{document}